\documentclass[11pt,thmsa]{amsart}

\usepackage[all]{xy}
\usepackage{color}
\usepackage{mathtools} 
\usepackage{fancybox}
\usepackage{amssymb}
\usepackage{hyperref}
\usepackage{comment}
\usepackage{graphicx}
\usepackage{faktor}

 \newtheorem{theorem}{Theorem}[section]
 \newtheorem*{thm*}{Theorem} 
 \newtheorem{lemma}[theorem]{Lemma}
 \newtheorem{proposition}[theorem]{Proposition}
 
 \newtheorem{corollary}[theorem]{Corollary}
 \newtheorem{definition}[theorem]{Definition}
  \theoremstyle{definition}
 \newtheorem{example}[theorem]{Example}
 \newtheorem{remark}[theorem]{Remark}

\newcommand{\Presym}{\mathsf{Pre}\textrm{-}\mathsf{Sym}}

\newcommand{\hor}{\mathrm{hor}}

\newcommand{\Fol}{K}
\newcommand{\Foliations}{\mathsf{Fol}}
\newcommand{\Hom}{\mathrm{Hom}}

\newcommand{\id}{\mathrm{id}}

\newcommand{\Lie}{\mathcal{L}}

\newcommand{\Diffop}{\mathsf{DO}}
\newcommand{\smooth}{\mathcal{C}}

\newcommand{\Koszul}{\mathcal{K}}
\newcommand{\Coder}{\mathrm{Coder}}

\newcommand{\MC}{\mathsf{MC}}

\newcommand{\cL}{\mathcal{L}}

\newcommand{\ZZ}{\ensuremath{\mathbb Z}}
\newcommand{\RR}{\ensuremath{\mathbb R}}

\newcommand{\pd}[1]{\frac{\partial}{\partial #1}} 

\newcommand{\lione}{$L_{\infty}[1]$-algebra }
\newcommand{\liones}{$L_{\infty}[1]$-algebras }

\begin{document}

\author{Florian Sch\"atz}
\email{florian.schaetz@gmail.com}
\address{University of Luxembourg, Mathematics Research Unit, 
Maison du Nombre
6, avenue de la Fonte 
L-4364 Esch-sur-Alzette, Luxembourg.
}
  
 \author{Marco Zambon}
\email{marco.zambon@kuleuven.be}
\address{KU Leuven, Department of Mathematics, Celestijnenlaan 200B box 2400, BE-3001 Leuven, Belgium.}


 \title[Deformations of pre-symplectic structures and the Koszul $L_{\infty}$ algebra]{Deformations of pre-symplectic structures \\
 and the Koszul $L_{\infty}$-algebra}

 \begin{abstract}
We study the deformation theory of pre-symplectic structures, i.e.~closed two-forms of fixed rank.
The main result is a parametrization of nearby deformations of a given pre-symplectic structure in terms
of an $L_\infty$-algebra, which we call the {\em Koszul $L_\infty$-algebra}.
This $L_\infty$-algebra is a cousin of the Koszul dg Lie algebra
associated to a Poisson manifold. In addition, we show that a quotient of the Koszul $L_{\infty}$-algebra
is isomorphic to the $L_\infty$-algebra which controls the deformations of the underlying characteristic foliation. Finally, we show that the infinitesimal deformations of pre-symplectic structures and of foliations are both obstructed.
\end{abstract}

\maketitle

\tableofcontents

\section*{\textsf{Introduction}}

This paper studies the deformation theory of closed 2-forms of a fixed rank.
Recall that a $2$-form $\eta$ on a manifold $M$ is called {\em pre-symplectic} if
\begin{itemize}
\item[i)] it is closed, and
\item[ii)] the kernel of the vector bundle map $\eta^\sharp: TM \to T^*M,\, v\mapsto \eta(v,\cdot)$ has constant rank. 
\end{itemize}
Pre-symplectic structures arise naturally on certain submanifolds of symplectic manifolds (as in the Dirac theory of constraints in mechanics) and as pullbacks of symplectic forms along submersions.
We denote  by  $\Presym^k(M)$ the set of all pre-symplectic structures on $M$, whose rank equals $k$.
We informally think of $\Presym^k(M)$ as some kind of space. Our main goal in this paper is to construct local parametrizations (`charts') for this space.

It is evident that in order to achieve this, we will need to take care simultaneously of the closedness condition and the rank condition. It is not hard to deal with each of these two conditions separately:
\begin{itemize}
\item[i)] closedness is the condition of lying in the kernel of a linear differential operator, the de Rham differential $d:\Omega^2(M)\to \Omega^3(M)$,
\item[ii)] fibrewise, the rank condition cuts out a (non-linear) fibre-subbundle of $\wedge^2 T^*M \to M$, for which one can construct explicit trivializations.
\end{itemize}
However, there is a certain tension between the two conditions:
The de Rham differential is in no obvious way compatible with the rank condition. Similarly, if one uses in ii) a naive trivialization for the rank condition, one loses control over the closedness condition.
\\
 
We therefore need a construction which addresses the closedness condition and the rank condition on equal footing. Given a pre-symplectic structure $\eta$, 
we make a choice of a complement of $\ker(\eta^\sharp)$ in $TM$, which determines
 a bivector field $Z$ (obtained by inverting the restriction of $\eta$  to the chosen complement). Using $Z$ we define a smooth, fiber-preserving but non-linear map
 $$F\colon (\text{neighborhood of the zero section in $\wedge^2 T^*M$}) \to   \wedge^2 T^*M.$$
The crucial fact about the transformation $F$ 
is that it 
 behaves well {\em both} with respect to the closedness and the rank condition!
Indeed:
\begin{itemize}
\item[i)] 
if the restriction of a (sufficiently small) 
$2$-form $\beta$ to the kernel of $\eta$ vanishes, then
$\eta + F(\beta)$ has the same rank as $\eta$, and vice versa {(see Theorem. \ref{theorem: almost Dirac structures})}.
\item[ii)]
the closedness of $F(\beta)$ is equivalent to  
an explicit equation for $\beta$ of the form
\begin{align}\label{MC} d\beta + \frac{1}{2}[\beta,\beta]_Z+ \frac{1}{6}[\beta,\beta,\beta]_Z =0.
\end{align}
\end{itemize}
Here $d$ is the de Rham differential, $[\cdot,\cdot]_Z$ and $[\cdot,\cdot,\cdot]_Z$ is a bilinear, respectively trilinear, map from $\Omega(M)$ to itself.
As the notation indicates, $[\cdot,\cdot]_Z$
and $[\cdot,\cdot,\cdot]_Z$ depend on $Z$.
In a more technical language, $d$, $[\cdot,\cdot]_Z$ and $[\cdot,\cdot,\cdot]_Z$
equip $\Omega(M)$ with the structure of an $L_\infty$-algebra (after a degree shift), Equation \eqref{MC}
is known as the corresponding {\em Maurer-Cartan equation}, and its solutions are called {\em Maurer-Cartan elements}. Our proofs rely on the construction of $L_{\infty}$-algebras out of commutative $BV_{\infty}$-algebras \cite{Koszul,Kravchenko}, and the map $F$ can be interpreted as the map on Maurer-Cartan elements induced by an $L_{\infty}$-isomorphism to the de Rham complex.

To complete our construction, we prove -- see Theorem \ref{theorem: construction - Koszul L-infty} -- that 
$$\Omega_\hor(M):=\{\beta \in \Omega(M) \, | \, \iota_X\beta=0 \text{ for all }X\in \ker(\eta^\sharp)\}$$
is closed under $d$, $[\cdot,\cdot]_Z$ and $[\cdot,\cdot,\cdot]_Z$, and hence inherits the structure of an $L_\infty$-algebra. We refer to it as the {\em Koszul $L_\infty$-algebra of $(M,\eta)$}, since $[\cdot,\cdot]_Z$ is defined by the same formula as the classical Koszul bracket for a Poisson bivector field, c.f.~\cite{Koszul}.
The corresponding Maurer-Cartan equation \eqref{MC} incorporates both the closedness {\em and} the rank condition which define $\Presym^k(M)$. Summing up our discussion,   
our main result is the construction of the following map (see Theorem \ref{theorem: main result}):
\begin{thm*}
There is an injective map
\begin{align}\label{intro: parametrization} \boxed{\{\textrm{small Maurer-Cartan elements of } (\Omega_\hor(M),d,[\cdot,\cdot]_Z,[\cdot,\cdot,\cdot]_Z)\} \;\;{\longrightarrow}\;\; \Presym^k(M)}
\end{align}
which bijects onto a ($\mathcal{C}^0$-)neighborhood of $\eta$.
\end{thm*}

Pre-symplectic structures have a rich geometry. One interesting feature is that each pre-symplectic structure $\eta$ induces a foliation on the underlying manifold, called the {\em characteristic foliation} of $\eta$, which is given by the kernel of $\eta$. This yields a map
$$\rho: \Presym^k(M) \to \{\textrm{foliations on } M\}.$$
We use the local parametrization \eqref{intro: parametrization}
of pre-symplectic structures of fixed rank
 to construct an algebraic model of this map. To be more precise, we show that a certain quotient of the Koszul $L_\infty$-algebra of $(M,\eta)$
is isomorphic to the $L_\infty$-algebra whose Maurer-Cartan equation encodes the deformations of the characteristic foliation, see Theorem \ref{theorem: characteristic foliation}.  

We finish the paper addressing the obstructedness problem: can every first order deformation be extended to a smooth curve of deformations? We show that the answer to this question is negative, both for deformations of pre-symplectic structures and of foliations. We do so by exhibiting counterexamples on the 4-dimensional torus, using the explicit form of the  $L_\infty$-algebras constructed previously.\\

Let us mention three important issues which we plan to address in follow-up papers:

\begin{itemize}

\item {\em Dirac-geometric interpretation and independence from auxiliary data:} 
Recall that a Dirac structure is a  maximal isotropic involutive subbundle of the Courant algebroid $TM\oplus T^*M$. Since $\eta$ is in particular a closed 2-form,   $\mathrm{graph}(\eta)$ is a Dirac structure. It turns out {\cite{SZDirac}} 
that parametrizing the deformations of $\mathrm{graph}(\eta)$  with the help of a suitable complement in $TM\oplus T^*M$ provides a geometric interpretation of the map \eqref{intro: parametrization}.

Recall also that the Koszul $L_{\infty}$-algebra of $(M,\eta)$  depends on a choice of auxiliary data. {The Dirac-geometric interpretation allows to apply the results of 
the recent preprint \cite{GMS} by M.~Gualtieri, M.~Matviichuk and G.~Scott to show that 
the Koszul $L_{\infty}$-algebra of $(M,\eta)$ is independent of these data up to $L_{\infty}$-isomorphisms.}

\item {\em Geometric vs. algebraic equivalences}: Isotopies act via pullback on the space $\Presym^k(M)$ of all pre-symplectic structures of fixed rank. This gives rise to an equivalence relation on $\Presym^k(M)$. On the other hand, the Koszul $L_\infty$-algebra of $(M,\eta)$ comes with the notion of gauge-equivalence on the set of Maurer-Cartan elements. We will show {\cite{SZpreequi}}  
 that these two notions of equivalence correspond to each other under the map \eqref{intro: parametrization}, assuming $M$ is compact.
 We resolved a problem of this type in our previous work \cite{Schaetz-Zambon2}.

\item {\em Relation to coisotropic submanifolds}: One of our motivations to develop the deformation theory of pre-symplectic structures is the relationship to coisotropic submanifolds, see \cite{OP,Schaetz-Zambon,Schaetz-Zambon2,LOTVcoiso}.
A submanifold $M$ of a symplectic manifold $(X,\omega)$ is {\em coisotropic} if the symplectic orthogonal to $TM$ inside $TX\vert_M$ is contained in $TM$. This condition guarantees that $\omega\vert_M$ is pre-symplectic and, in fact, every pre-symplectic form can be obtained this way. This hints at a tight relationship between the deformation theory of coisotropic submanifolds and the deformation theory of pre-symplectic structures.
We will show {\cite{SZcoisopre}} 
that this is indeed the case. On the geometric level, we extend previous work by Ruan, cf.~\cite{Ruan}. On the algebraic level, we prove that the Koszul $L_\infty$-algebra
of $(M,\eta)$ is homotopy equivalent -- or quasi-isomorphic -- to the {$L_{\infty}$-algebra} 
 that controls the simultaneous deformations of pairs consisting of a symplectic structure and a coisotropic submanifold, see \cite{Bandiera-Manetti,FZgeo}.
\end{itemize}

\bigskip
\noindent{\bf{Structure of the paper:}}

Section \ref{section: pre-symplectic structures} sets the stage
by introducing the relevant deformation problem, and by establishing basic geometric and algebraic facts related to pre-symplectic structures. In Section \ref{subsection: symplectic} we discuss the toy-example of symplectic structures in the framework which we generalize to arbitrary pre-symplectic structures in the rest of the paper.

Section \ref{section: linear algebra} is devoted to the rank condition. We discuss in detail special submanifold charts for the space of skew-symmetric bilinear forms of a fixed rank.

Section \ref{section: Koszul Lie 3-algebra} is the heart of the paper. Here, we construct the Koszul $L_\infty$-algebra associated to a pre-symplectic structure $\eta$, and prove that its Maurer-Cartan equation encodes the deformations of $\eta$ inside the space of pre-symplectic structures of fixed rank. We located the proofs concerning $L_\infty$-algebras in Subsection \ref{subsection: proofs L-infty}.
Subsection \ref{section: examples} illustrates our approach with examples.

Section \ref{section: characteristic foliation} links the deformation theory of pre-symplectic structures to the deformation theory of their characteristic foliations.
On the algebraic level, this relationship is quite straightforward: we show that the $L_\infty$-algebra, which encodes the deformations of the foliation underlying $\eta$, arises as a quotient of the Koszul $L_\infty$-algebra by an $L_\infty$-ideal.

Section \ref{section: obstructions} shows that the infinitesimal deformations of   pre-symplectic structures as well as foliations are obstructed, using properties of the $L_\infty$-algebras obtained in Section \ref{section: Koszul Lie 3-algebra} and  \ref{section: characteristic foliation}.
\\

\paragraph{\bf Acknowledgements:}
We thank Marco Gualtieri and Mykola Matviichuck for discussions about
their work \cite{GMS} (joint with G. Scott) and Donald Youmans for correcting an inaccuracy in a previous version of this paper.
F.S. thanks Ruggero Bandiera for his generous help concerning Koszul brackets, and especially for the discussions leading to the proof of Proposition \ref{proposition: extended Koszul}.
M.Z. acknowledges partial support by Pesquisador Visitante Especial grant 88881.030367/2013-01 (CAPES/Brazil), by IAP Dygest (Belgium), 
the long term structural funding -- Methusalem grant of the Flemish Government.  
\\

\paragraph{\bf Convention:}
Let $V$ be a $\mathbb{Z}$-graded vector space, that is we have a decomposition $V=\bigoplus_{k\in \mathbb{Z}}V_k$.
For $r\in \mathbb{Z}$, we denote by $V[r]$ the $\mathbb{Z}$-graded vector space whose component in degree $k$ is $V_{k+r}$.

For $V$ a graded vector space, we denote by $\odot^nV$ the space of coinvariants of $V^{\otimes n}$ under the action by the symmetric group $S_n$ via Koszul signs.

 \addtocontents{toc}{\protect\mbox{}\protect}
\section{\textsf{Pre-symplectic structures and their deformations}}\label{section: pre-symplectic structures}

\subsection{Deformations of pre-symplectic structures}
\label{subsection: pre-symplectic structures}

In this section we set up the deformation problem which we study in this article. Throughout the discussion, $M$ denotes a smooth manifold.

\begin{definition}\label{definition: pre-symplectic}
A $2$-form $\eta$ on $M$ is called \emph{pre-symplectic} if 
\begin{enumerate}
\item $\eta$ is closed and
\item the vector bundle map $\eta^\sharp: TM \to T^*M, v\mapsto \iota_v\eta = \eta(v,\cdot)$ has constant rank.
\end{enumerate}
In the following, we refer to the rank of $\eta^\sharp$ as the rank of $\eta$.
\end{definition}

\begin{definition}
A \emph{pre-symplectic} manifold is a pair $(M,\eta)$ consisting of a manifold $M$
and a pre-symplectic form $\eta$ on $M$.

We denote the space of all pre-symplectic structures on $M$ by 
$\Presym(M)$ and the space of all pre-symplectic structures of rank $k$
by $\Presym^k(M)$.
\end{definition}

\begin{remark}
Given a pre-symplectic manifold $(M,\eta)$, the fibrewise kernels of 
$\eta^\sharp$ assemble into a vector subbundle of $TM$, which we denote by $K$:
$$K:=\ker(\eta^\sharp).$$
Since $\eta$ is closed, $K$ is involutive, hence gives
rise to a foliation of $M$.
Recall that associated to any foliation, one has the corresponding foliated
de~Rham complex, which we denote by
$\Omega(K):=(\Gamma(\wedge K^*),d_\Fol).$
Restriction of ordinary differential forms on $M$ to sections of $K$ yields a surjective chain map
$r: \Omega(M) \to \Omega(K)$.
We denote the kernel of $r$ by $\Omega_\hor(M)$. It coincides with the multiplicative ideal
in $\Omega(M)$ generated by all the section of the annihilator $K^\circ \subset T^*M$ of $K$. We have the following exact sequence of complexes
\begin{equation}\label{eq:ses}
\xymatrix{
0 \ar[r] &\Omega_\hor(M) \ar[r] & \Omega(M) \ar[r]^r & \Omega(K)\ar[r] & 0
}.
\end{equation}
We denote the cohomology of $\Omega_\hor(M)$ by $H_\hor(M)$, and the cohomology of $\Omega(K)$
by $H(K)$.
\end{remark}

We next compute the formal tangent space to
$\Presym^k(M)$ at a pre-symplectic form $\eta$.
 
\begin{lemma}\label{lemma: tangent space to pre-sym}
Let $(\eta_t)_{t\in [0,\varepsilon)}$ be a one-parameter family of pre-symplectic forms on $M$
of fixed rank $k$ with $\eta_0=\eta$.
Then $\frac{d}{dt}\vert_{t=0}\eta_t$ is a closed $2$-form which lies in 
$\Omega^2_\hor(M)$.
\end{lemma}

\begin{proof}
That $\frac{d}{dt}\vert_{t=0}\eta_t$ is closed is straight-forward.

Concerning the second claim, let $X$ and $Y$ be two arbitrary smooth sections
of $K$. Since $(K_t=\ker(\eta_t^\sharp))_{t\in [0,\varepsilon)}$ assemble into a smooth vector bundle $\tilde{K}$ over
$M\times [0,\varepsilon)$, we can find smooth extensions of $X$ and $Y$ to
sections of $\tilde{K}$. Denote these extensions by $X_t$ and $Y_t$.
By definition, we have
$$ 0=\eta_t(X_t,Y_t)$$
for all $t\in [0,\varepsilon)$. Differentiation with respect to $t$ yields
$$  0= \left(\frac{d}{dt}\vert_{t=0}\eta_t\right)(X,Y) + \eta(\frac{d}{dt}\vert_{t=0}X_t, Y) + \eta(X,\frac{d}{dt}\vert_{t=0}Y_t) = \left(\frac{d}{dt}\vert_{t=0}\eta_t\right)(X,Y).$$
\end{proof}

\begin{remark}\label{rem:tpres}
Lemma \ref{lemma: tangent space to pre-sym} tells us that
we have the following identification of the formal tangent space to $\Presym^k(M)$ at $\eta$:
$$ T_\eta \left(\Presym^k(M)\right)  \cong \{\alpha \in \Omega^2(M) \textrm{ closed}, r(\alpha)=0\} = Z^2(\Omega_\hor(M)).$$
\end{remark}

\subsection{Bivector fields induced by pre-symplectic forms}\label{section: bivec}

For later use,  
we present basic results on the geometry of bivector fields that arise from  pre-symplectic forms, after one makes a choice of complement to the kernel. For any bivector field $Z$, we denote by $\sharp \colon T^*M\to TM$ the map $\xi\mapsto \iota_{\xi}Z=Z(\xi,\cdot)$. The Koszul bracket of $1$-forms  associated to $Z$ is 
\begin{equation}
\label{eq:KoszulBracket1f}
[\xi_1,\xi_2]_{Z}=\iota_{\sharp \xi_1}d\xi_2-\iota_{\sharp \xi_2}d\xi_1+d\langle Z, \xi_1\wedge \xi_2\rangle. 
\end{equation}
In case $Z$ is Poisson, $\sharp$ and $[\cdot,\cdot]_Z$ make $T^*M$ into a Lie algebroid. In general, there is an induced bracket on smooth functions given by $\{f,g\}_Z=(\sharp df)(g)$.
In the following, we will make repeated use of the fact that the Koszul bracket has the following derivation property:
$$[\xi_1, f\xi_2]_Z= f[\xi_1,\xi_2]_Z + Z(\xi_1,df)\xi_2 $$
where $\xi_1$, $\xi_2\in \Omega^1(M)$ and $f\in \smooth^\infty(M)$.

 We start with a lemma about general bivector fields, which in the Poisson case reduces to the fact that $\sharp$ is bracket-preserving:
\begin{lemma}\label{lem:bivector}
For any bivector field $Z$ on $M$, and for all $\xi_1,\xi_2\in \Omega^1(M)$ we have
\begin{equation}\label{eq:sharpnothom}
[\sharp \xi_1,\sharp \xi_2]=\sharp [\xi_1,\xi_2]_{Z}-\frac{1}{2}\iota_{\xi_2}\iota_{\xi_1}[Z,Z].
\end{equation}
 \end{lemma}
\begin{proof}
By the derivation property of the Koszul bracket, we may assume that $\xi_i$ is exact for $i=1,2$.
We have $\langle [\sharp df_1,\sharp df_2],df_3\rangle=
 [\sharp df_1,\sharp df_2](f_3)=\{f_1,\{f_2,f_3\}\}-\{f_2,\{f_1,f_3\}\}$,
 using in the first equality the definition of the Lie bracket as a commutator. Further
 $\langle \sharp[df_1,df_2]_{Z},df_3\rangle= 
  \langle \sharp d\{f_1,f_2\},df_3\rangle=
  \{\{f_1,f_2\},f_3\}$. Now use the well-known fact that $\frac{1}{2}[Z,Z]$ applied to $df_1\wedge df_2\wedge df_3$ equals the Jacobiator  $\{\{f_1,f_2\},f_3\}\}+c.p.$
\end{proof}

{\begin{remark}\label{rem:nokkk}
Let $Z$ be a constant rank bivector field, and denote by $G$ the image of $\sharp$. Since $G^{\circ}$ is the kernel of $\sharp$, it is straightforward to check that $[\xi_1,\xi_2]_{Z}=0$ for all $\xi_1,\xi_2\in\Gamma(K^*)$. Lemma \ref{lem:bivector} immediately implies that, for any splitting $TM=K\oplus G$, we have
$[Z,Z]\in \Gamma(\wedge^3 G)\oplus \Gamma(\wedge^2 G\otimes K)$.
\end{remark} 
}

Let $\eta\in \Omega^2(M)$ is be a pre-symplectic structure on $M$, and denote its kernel by $K$.
Let us fix a complementary subbundle $G$, so $TM=K\oplus G$. Define $Z$ to be the bivector field on $M$ determined by
$ Z^{\sharp} = -(\eta\vert_G^\sharp)^{-1}$. Clearly $Z$ is a constant rank bivector field, and the image of $ Z^{\sharp}$ is $G$.\\

Together with  Remark \ref{rem:nokkk}, the following Lemma implies that  $[Z,Z]\in \Gamma(\wedge^2G\otimes K)$.
\begin{lemma}\label{lem:noggg}
$[Z,Z]$ has no component in $\wedge^3 G$. \end{lemma}
\begin{proof}
Working locally, we may assume that we have a surjective submersion $p\colon M\to M/K$ where $M/K$ is the quotient of $M$ by the foliation integrating $K$. By pre-symplectic reduction, there is a unique symplectic form $\Omega$ on $M/K$ such that $$p^*\Omega=\eta.$$ Denote by $\Pi$ the Poisson bivector field on $M/K$ determined by $ \Pi^{\sharp} = -(\Omega^\sharp)^{-1}$. Under the decomposition $TM=K\oplus G$, the only component of $\eta$ is $\eta|_{\wedge^2G}\in \Gamma(\wedge^2 G^*)$, whose negative inverse is $Z\in \Gamma(\wedge^2G)$. By the above equation, $Z$ projects to $\Pi$ under $p$. Consequently, the trivector
field $[Z,Z]$ projects onto $[\Pi,\Pi]=0$, finishing the proof.
\end{proof}

We wish to understand the Lie bracket of vector fields $ Z^{\sharp} \xi$ where $\xi$ is a $1$-form.
Clearly  $K^*=G^{\circ}$ is the kernel of $ Z^{\sharp}$. Hence it is sufficient to assume that $\xi$ be a section of $G^*$.
\begin{lemma}\label{lem:morG*}
 For all $\xi_1,\xi_2\in \Gamma(G^*)$, Equation \eqref{eq:sharpnothom}
expresses $[ Z^{\sharp} \xi_1, Z^{\sharp} \xi_2]$ in terms of the decomposition $TM=G\oplus K$. In particular,
$$Z^{\sharp}[\xi_1,\xi_2]_Z=\mathrm{pr}_G([ Z^{\sharp} \xi_1, Z^{\sharp} \xi_2]).$$
\end{lemma}
\begin{proof}
The first term on the right-hand side of Equation \eqref{eq:sharpnothom} lies in $G=\mathrm{image}( Z^{\sharp})$.
Since $[Z,Z]\in \Gamma(\wedge^3 TM)$ has no component in $\wedge^3 G$  by Lemma \ref{lem:noggg}, the last term on the right-hand side of Equation \eqref{eq:sharpnothom} lies in $K$.
\end{proof}

We finish  with:
\begin{lemma}\label{lem:g*g*g*}
$\Gamma(G^*)$ is closed under the Koszul bracket.
\end{lemma}
\begin{proof}
It suffices to show that the Koszul bracket of any elements from a frame for $G^*=K^\circ$ lies again in $\Gamma(K^\circ)$. We use the same notation as in the proof of Lemma \ref{lem:noggg}.
If we pick a system of coordinates $y_1,\dots,y_r$ on $M/K$ (we work locally), the $1$-forms
$d(p^*y_1),\dots, d(p^*y_r)$ constitute a frame of $K^\circ$.
We have
$$ [d(p^*y_i),d(p^*y_j)]_{Z}= d(\langle Z,d(p^*y_i)\wedge d(p^*y_j)\rangle)= dp^*(\langle\Pi,y_i\wedge y_j\rangle)$$
and this is clearly an element of $\Gamma(K^\circ)$.
\end{proof}

 \subsection{The Koszul bracket and deformations of symplectic structures}\label{subsection: symplectic}

In this subsection 
we describe how the deformations of symplectic structures, once the Poisson geometry point of view is taken, can be described by means of the Koszul bracket.
 This approach will be generalized to arbitrary pre-symplectic structures in Section \ref{section: Koszul Lie 3-algebra}, relying on some linear algebra developed in Section \ref{section: linear algebra},
see in particular Remark \ref{rem:symexp}. Of course, the nearby deformations of a symplectic structure $\omega$ can also be described as $\omega+\alpha$ for $\alpha$  small 2-forms satisfying $d\alpha=0$, but this straightforward description does not extend to the pre-symplectic case.

Let $\pi$ be a Poisson bivector field on $M$.
There is a unique extension of the Koszul bracket $[\cdot,\cdot]_\pi$ -- defined on $1$-forms by formula \eqref{eq:KoszulBracket1f} -- to all differential forms such that the extension is graded skew-symmetric, the Leibniz rule for $d$ is satisfied, and  the derivation property with respect to the wedge product is satisfied 
(see Remark \ref{rem:Kos2der} for the precise formuale).
 
 We recall the following algebraic facts:
\begin{lemma}\label{lem:strict} Let $\pi$ be any Poisson structure on a manifold $M$.
\begin{itemize}
\item [(a)]
The following is a strict morphism of dg Lie algebras:
$$I:=\wedge\pi^{\sharp}\colon (\Omega(M),d,[\cdot,\cdot]_{\pi})\to (\mathfrak{X}^\mathrm{multi}(M),-[\pi,\cdot],[\cdot,\cdot]).$$
\item [(b)]The target dg Lie algebra $\mathfrak{X}^\mathrm{multi}(M)$  governs the deformations of the Poisson structure $\pi$:  Poisson structures nearby $\pi$ are given by bivector fields $\pi-\tilde{\pi}$, where $\tilde{\pi}$ satisfies the Maurer-Cartan equation
$$ -[\pi,\tilde{\pi}] + \frac{1}{2}[\tilde{\pi},\tilde{\pi}]=0.$$
\end{itemize}
\end{lemma}
\begin{proof}
(a) Recall that the anchor of the cotangent Lie algebroid $\pi^{\sharp}\colon T^*M\to TM$ is a Lie algebroid morphism. Hence the pullback $\wedge(\pi^{\sharp})^*\colon \Omega(M)\to \mathfrak{X}^{\mathrm{multi}}(M)$ relates the Lie algebroid differential $d$ (the de~Rham differential) to $d_{\pi}:=[\pi,\cdot]$. Since $(\pi^{\sharp})^*=-\pi^{\sharp}$, it follows that $\wedge\pi^{\sharp}$ maps   $d$  to $-d_{\pi}$. Further, since  $\pi^{\sharp}$ is a Lie algebroid morphism, $\wedge\pi^{\sharp}$ preserves (Schouten) brackets.

(b) This follows from $[\pi-\tilde{\pi},\pi-\tilde{\pi}]= -2[\pi, \tilde{\pi}]+[\tilde{\pi},\tilde{\pi}]$.
\end{proof}

 Suppose that $\pi$ is invertible, and denote by $\omega$ the corresponding symplectic structure, determined by $-\omega^{\sharp}=(\pi^{\sharp})^{-1}$. Denote by $\mathcal{I}_\pi$ the tubular neighborhood of $M\subset \wedge^2 T^*M$ consisting of those bilinear forms $\beta$ such that $\mathrm{id}_{TM}+ \pi^\sharp\beta^\sharp$
is invertible. The following lemma takes the point of view of Poisson geometry to describe the symplectic structures nearby $\omega$, in the sense that instead of deforming $\omega$ directly, it deforms $\pi$. Diagrammatically:
$$
\xymatrix{
\{\text{symplectic forms near $\omega$}\} \ar@{<->}[r]^{\text{inversion}}& \{\text{Poisson structures near $\pi$}\} \ar@{<->}[d]\\
\{\text{small $\beta$ s.t. $ d\beta +\frac{1}{2}[\beta,\beta]_\pi=0$}\}
\ar@{<->}[r]^{I}&
\{\text{small $\tilde{\pi}$ s.t. $-[\pi, \tilde{\pi}]+\frac{1}{2}[\tilde{\pi},\tilde{\pi}]
=0$}\}
 \\
}
$$

\begin{lemma}\label{lem:phibinv}
Let $\omega$ be a symplectic structure on $M$ with corresponding Poisson structure $\pi$.
There is a bijection between
\begin{itemize}
\item $2$-forms $\beta \in \Gamma(\mathcal{I}_\pi)$ such that the equation $ d\beta +\frac{1}{2}[\beta,\beta]_\pi=0$ holds.
\item symplectic forms nearby $\omega$ (in the $\mathcal{C}^0$ sense).
\end{itemize}

The bijection maps $\beta$ to the symplectic form with sharp map $\omega^\sharp + \beta^\sharp(\mathrm{id}_{TM}+\pi^\sharp\beta^\sharp)^{-1}$.
\end{lemma}

\begin{proof}
It is well-known that under the correspondence between non-degenerate $2$-forms and non-degenerate bivector fields, the closeness of the $2$-form $\omega$ corresponds to the Poisson condition for $\pi$. 
By Lemma \ref{lem:strict} (b), the Poisson structures  nearby $\pi$ are given by $\pi-\tilde{\pi}$ where $\tilde{\pi}$ satisfies the Maurer-Cartan equation of the dg Lie algebra $\mathfrak{X}^\mathrm{multi}(M)$.

Since   $\pi$ is non-degenerate, the map $I$   is an isomorphism between differential forms and multivector fields.
The 2-form $\beta:=I^{-1}(\tilde{\pi})$ satisfies the Maurer-Cartan equation $ d\beta +\frac{1}{2}[\beta,\beta]_\pi=0$  by Lemma \ref{lem:strict} (a). Notice that  $\pi-\tilde{\pi}=\pi-I\beta$ has sharp map 
\begin{equation}\label{eq:pibetasharp}
(\mathrm{id}_{TM}+\pi^{\sharp}\beta^\sharp)\pi^\sharp\colon T^*M\to TM,  
\end{equation}
 so it is non-degenerate if{f} $\beta \in \Gamma(\mathcal{I}_\pi)$. Hence we obtain a bijection between  2-forms $\beta$ as in the statement on one side, and symplectic forms
corresponding to $\pi-I\beta$ on the other side.

The latter can be described as follows:
the sharp map is
\begin{eqnarray*}
-[(\pi-I\beta)^{\sharp}]^{-1}&=&
-(\pi^\sharp)^{-1}(\mathrm{id}_{TM}+\pi^{\sharp}\beta^\sharp)^{-1}\\
&=&\omega^\sharp((\mathrm{id}_{TM} + \pi^\sharp\beta^\sharp) - \pi^\sharp\beta^\sharp)(\mathrm{id}_{TM}+\pi^\sharp\beta^\sharp)^{-1}\\
&=&\omega^\sharp + \beta^\sharp(\mathrm{id}_{TM}+\pi^\sharp\beta^\sharp)^{-1},
\end{eqnarray*}
where in the first equality we used  \eqref{eq:pibetasharp}.
  \end{proof}

 \addtocontents{toc}{\protect\mbox{}\protect}
\section{\textsf{Parametrizing skew-symmetric bilinear forms}}\label{section: linear algebra}

In this section, we discuss the rank condition on pre-symplectic structures.
 Since we postpone the discussion of integrability
to Section \ref{subsection: Koszul for pre-symplectic}, everything boils down
to linear algebra.
In Section \ref{subsection: Dirac parametrization} we introduce a certain local parametrization of skew-symmetric bilinear forms,
which is inspired by Dirac geometry. 
 
\subsection{A parametrization inspired by Dirac geometry}\label{subsection: Dirac parametrization}

Let $V$ be a finite-dimensional, real vector space. Fix $Z\in \wedge^2 V$ a bivector, which can be encoded by the linear map  $$Z^\sharp\colon V^* \to V, \quad \xi \mapsto \iota_\xi Z = Z(\xi,\cdot).$$

\begin{definition}
We denote by $\mathcal{I}_Z$ the open neighborhood of
$0 \subset \wedge^2 V^*$ consisting of those elements $\beta$
for which the map $\mathrm{id} + Z^\sharp \beta^\sharp\colon V\to V$ is invertible.
\end{definition}

We consider the map  $F \colon \mathcal{I}_Z\to \wedge^2 V^*$ determined by 
\begin{equation}\label{eq:alphabeta}
(F (\beta))^\sharp = \beta^\sharp(\id + Z^{\sharp}\beta^\sharp)^{-1}.
\end{equation}
This map is clearly non-linear and smooth.

\begin{remark}\label{rem:alfabeta}
\begin{itemize}
\item[(i)]
We have $\mathrm{ker}(\beta)=\mathrm{ker}(F(\beta))$. The inclusion ``$\subset$'' follows directly from Equation \eqref{eq:alphabeta}, using the fact that $\id+Z^{\sharp}\beta^{\sharp}$ preserves $\mathrm{ker}(\beta)$. Further, since  $\id+Z^{\sharp}\beta^{\sharp}$  is an isomorphism, the dimensions of $\mathrm{ker}(\beta)$ and $\mathrm{ker}(F(\beta))$ are the same.
\item[(ii)]
$F\colon \mathcal{I}_Z \to \wedge^2V^*$ bijects onto $\mathcal{I}_{-Z}$, with inverse $\alpha\mapsto \alpha^\sharp(\id - Z^{\sharp}\alpha^\sharp)^{-1}$.   Indeed $$\id-Z^{\sharp}(F(\beta))^{\sharp}=\id-Z^{\sharp}\beta^\sharp(\id + Z^{\sharp}\beta^\sharp)^{-1}=(\id+Z^{\sharp}\beta^{\sharp})^{-1},$$ showing that $$(F(\beta))^\sharp(\id - Z^{\sharp}(F(\beta))^\sharp)^{-1}
=\beta^\sharp.$$
\item [(iii)] {Denote by $G$ the image of $Z^{\sharp}$ (so $Z$ can be regarded as a non-degenerate element of $\wedge^2 G$).
For any $\beta \in  \wedge^2V^*$, we have  $\beta\in  \mathcal{I}_Z$ if{f} $\sigma:=\beta|_{\wedge^2G}$ satisfies the condition that
$\mathrm{id_G} + Z^\sharp \sigma^\sharp\colon G\to G$ is invertible. This follows immediately writing $\mathrm{id_V} + Z^\sharp \beta^\sharp\colon V\to V$ as a block matrix in terms of any decomposition $V=K\oplus G$, and noticing that it is a lower triangular matrix.}
\end{itemize}
\end{remark}
$F$ is a diffeomorphism from $\mathcal{I}_{Z}$ to $\mathcal{I}_{-Z}$, which keeps the origin fixed.
We now use this transformation to construct submanifold charts for the space of skew-symmetric bilinear forms on $V$ of some fixed rank.

\begin{definition}
The \emph{rank} of an element $\eta \in \wedge^2 V^*$ is the rank
of the linear map $\eta^\sharp: V \to V^*$.
We denote the space of $2$-forms on $V$ of rank $k$ 
by $(\wedge^2V^*)_k$.
\end{definition}

Assume from now on that $\eta \in \wedge^2 V^*$ is of rank $k$.
We fix a subspace $G\subset V$, which is complementary to the kernel $K=\ker(\eta^\sharp)$.
Let
$ r: \wedge^2 V^* \to \wedge^2 K^*$ be the restriction map; we have the natural identification
$\ker(r) \cong \wedge^2 G^* \oplus (G^*\otimes K^*)$.
Since the restriction of $\eta$ to $G$ is non-degenerate, there is a unique
element $Z \in \wedge^2 G \subset\wedge^2 V$ determined by the requirement that
$$ Z^\sharp: G^*\to G, \quad \xi \mapsto \iota_\xi Z=Z(\xi,\cdot)$$
equals $-(\eta\vert_G^\sharp)^{-1}$.  

\begin{definition}\label{definition: Dirac exponential map}
The Dirac exponential map $\exp_{\eta}$ of $\eta$ (and for fixed $G$) is the mapping
$${\exp_{\eta}}\colon \mathcal{I}_Z \to \wedge^2 V^*, \quad \beta \mapsto \eta + 
F(\beta).$$ 
\end{definition}

\begin{remark}\label{rem:symexp}
When $\eta$ is non-degenerate, the construction of a nearby symplectic  form out of a small 2-form $\beta$ carried out in Subsection \ref{subsection: symplectic} agrees with $\exp_{\eta}(\beta)$, the image of $\beta$ under the Dirac exponential map.
This is clear from Lemma \ref{lem:phibinv} and Equation \eqref{eq:alphabeta}, and further  gives a justification for Definition \ref{definition: Dirac exponential map}.
\end{remark}

The following theorem asserts that $\exp_{\eta}$ is a submanifold chart for $(\wedge^2 V^*)_k \subset \wedge^2 V^*$.

\begin{theorem}\label{theorem: almost Dirac structures}
\hspace{0cm}
\begin{enumerate}
\item[(i)] Let $\beta \in \mathcal{I}_Z$. Then 
 $\exp_{\eta}(\beta)$ lies in $(\wedge^2 V)_k$ if, and only if, $\beta$ lies in $\ker(r)=(K^*\otimes G^*)\oplus \wedge^2G^*$.

\item[(ii)] 
Let $\beta = (\mu,\sigma) \in \mathcal{I}_Z\cap ((K^*\otimes G^*)\oplus \wedge^2G^*)$. Then  $\exp_{\eta}(\beta)$ is the unique skew-symmetric bilinear form on $V$ with the following properties:
\begin{itemize}
\item  its restriction to $G$ equals {$(\eta+F(\sigma))|_{\wedge^2G}$}
\item its kernel is the graph of the map $Z^\sharp \mu^\sharp= -(\eta\vert_G^{\sharp})^{-1}\mu^\sharp: K\to G$.
\end{itemize}

\item[(iii)]
The Dirac exponential map $\exp_{\eta}: \mathcal{I}_Z \to \wedge^2 V^*$ restricts to a 
diffeomorphism 
$$\mathcal{I}_Z \cap (K^*\otimes G^*)\oplus \wedge^2G^* \stackrel{\cong}{\longrightarrow}
{\{\eta' \in (\wedge^2 V^*)_k|\; \ker(\eta') \text{ is transverse to $G$}\}}
$$
onto an open neighborhood of $\eta$  in  $(\wedge^2 V^*)_k$.
\end{enumerate}
\end{theorem}

To prove the theorem we need a technical lemma.
\begin{lemma}\label{lem:kernel}
For  any $\beta \in \mathcal{I}_Z$ we have
$$\ker(\exp_{\eta}(\beta))= \text{image of the restriction of  $\id + Z^\sharp \beta^\sharp_m$ to $\ker(\beta_K)\subset K$}.$$ Here we
 use the notation
$$\beta = \beta_K + \beta_m + \beta_G \in \wedge^2 K^*\oplus (K^*\otimes G^*) \oplus \wedge^2 G^*.$$
\end{lemma}

\begin{proof}  For all $w\in V$ we have
\begin{align}
w\in \ker(\exp_{\eta}(\beta))  \quad &\Longleftrightarrow \quad \eta^{\sharp}w=-\beta^\sharp(\id + Z^{\sharp}\beta^\sharp)^{-1}w \nonumber\\
\quad &\Longleftrightarrow \quad \eta^{\sharp}(\id + Z^{\sharp}\beta^\sharp)v=-\beta^\sharp v \nonumber, \qquad \textrm{where $v:=(\id + Z^{\sharp}\beta^\sharp)^{-1}w$},\\
\quad & \Longleftrightarrow \quad \eta^{\sharp}v=-\eta^{\sharp}Z^{\sharp}\beta^\sharp v-\beta^\sharp v. \label{eq:separate}
\end{align}
The endomorphism $\eta^{\sharp}Z^\sharp$ has kernel $K$ and equals $-\mathrm{id}_G$ on $G$, hence
with respect to the decomposition $V=K\oplus G$,
the endomorphism 
$-\eta^{\sharp}Z^\sharp\beta^\sharp$ reads
$\left(\begin{array}{cc}0 &  0 \\
(\beta_m^{\sharp})|_K & \beta_G^{\sharp}\end{array}\right)$, and the endomorphism on the  right-hand side of Equation \eqref{eq:separate} reads 
$-\left(\begin{array}{cc}\beta_K^{\sharp} & (\beta_m^{\sharp})|_G\\ 0 &  0\end{array}\right)$ and, in particular, takes values in $K^*$. In contrast, the endomorphism $\eta^{\sharp}$ on the left-hand side of  Equation \eqref{eq:separate} takes values in $G^*$. Hence both sides have to vanish, and Equation \eqref{eq:separate} is equivalent to the condition  $v\in \ker(\beta_K)=\{v \in K \, : \, \iota_v\beta_K=0\}$,
where we used $\ker(\eta)=K$.\end{proof}

\begin{proof}[Proof of Theorem \ref{theorem: almost Dirac structures}]
We use the decomposition of $\beta$ as in Lemma \ref{lem:kernel}.

(i) By Lemma \ref{lem:kernel}, $\exp_{\eta}(\beta)$ lies in $(\wedge^2 V)_k$  if{f} $\beta_K=0$.

(ii) For  any $\beta \in \mathcal{I}_Z$,
 Lemma \ref{lem:kernel} implies that the second property on the kernel of $\exp_{\eta}(\beta)$
is satisfied, so we only have to prove the first property.
Since $Z^{\sharp}$ has image $G$ and kernel $K$, the map $\id + Z^{\sharp}\beta^\sharp$ sends $G$ isomorphically into itself, and
 its restriction to $G$ equals $(\id_G + Z^{\sharp}\beta_G^\sharp)$.
Hence $$(F(\beta))^{\sharp}|_G=
 \beta^\sharp(\id + Z^{\sharp}\beta^\sharp)^{-1}|_G=
 \beta^\sharp(\id_G + Z^{\sharp}\beta_G^\sharp)^{-1}.$$
Composing with the projection $V^*=K^*\oplus G^*\to G^*$ we obtain
 $\beta_G^\sharp(\id_G + Z^{\sharp}\beta_G^\sharp)^{-1}\colon G\to G^*$,
 so $F(\beta)(Y_1,Y_2)=F(\beta_G)(Y_1,Y_2)$ for all $Y_1,Y_2\in G$. This holds for  any $\beta \in \mathcal{I}_Z$, in particular also when $\beta_K=0$.

(iii) 
{
By item (i) and Remark \ref{rem:alfabeta}  (ii), the map $\exp_{\eta}$ sends $\mathcal{I}_Z \cap (K^*\otimes G^*)\oplus \wedge^2G^*$ bijectively onto 
$\{\eta' \in (\wedge^2 V^*)_k|\; \eta'-\eta\in \mathcal{I}_{-Z}\}$.
  Take any $\eta'\in (\wedge^2 V^*)_k$.
By Remark \ref{rem:alfabeta} (iii),
$\eta'-\eta\in \mathcal{I}_{-Z}$  is equivalent to $Id_G-Z^{\sharp}
(\eta'|_G-\eta|_G)^\sharp\colon G\to G$ being invertible. This in turn is equivalent to
 $(\eta'|_G)^\sharp$ being invertible, i.e. to $\ker(\eta')\oplus G=V$, due to
$$(\eta'|_G)^\sharp=(\eta|_G)^\sharp+(\eta'|_G-\eta|_G)^\sharp=(\eta|_G)^\sharp[Id_G-Z^{\sharp}
(\eta'|_G-\eta|_G)^\sharp]$$ and the invertibility of $(\eta|_G)^\sharp$.
}
\end{proof}

\begin{remark}\label{rem:parametrize}
We notice that the construction of $\exp_{\eta}$ can be readily extended to the case of vector bundles. In particular, given a pre-symplectic manifold $(M,\eta)$, the choice of a complementary subbundle $G$
to the kernel $K$ of $\eta$ yields a fibrewise map
$$\exp_{\eta}: (K^*\otimes G^*)\oplus (\wedge^2 G^*) \to \wedge^2 T^*M,$$
which maps the zero section to $\eta$, and an open neighborhood thereof into the space of $2$-forms of rank equal to that of $\eta$.
As a consequence, we can parametrize deformations of $\eta$
inside $\Presym^k(M)$
by sections $(\mu, \sigma)\in \Gamma(K^*\otimes G^*)\oplus \Gamma(\wedge^2 G^*) \cong \Omega_\hor^2(M)$ which are sufficiently close
to the zero section, and which satisfy
$$d ((\exp_{\eta})_*(\mu,\sigma))=0,$$
with $d$ the de~Rham differential.
In Section \ref{subsection: Koszul for pre-symplectic} we will show that the latter integrability condition can be rephrased in terms of an \lione structure on $\Omega(M)[2]$.
\end{remark}

  \addtocontents{toc}{\protect\mbox{}\protect}
\section{\textsf{The Koszul $L_\infty$-algebra}}\label{section: Koszul Lie 3-algebra}

In this section we introduce the Koszul $L_\infty$-algebra
of a pre-symplectic manifold $(M,\eta)$. This $L_\infty$-algebra lives on $\Omega_\hor(M)$ (with a certain shift in degrees),
and the zero set of its Maurer-Cartan equation parametrizes a neighborhood of $\eta$ in $\Presym^k(M)$.

\subsection{An $L_\infty$-algebra associated to a bivector field}
\label{subsection: Koszul for bivectors}

In this subsection, we introduce an $L_\infty$-algebra, which is naturally attached to some bivector field $Z$ on a manifold $M$. The statements made in this subsection will be proven in {Section
\ref{subsection: proofs L-infty} below.

Recall that there are two basics operations on differential forms, which one can associate to a multivector field $Y\in \Gamma(\wedge^k TM)$:
contraction $\iota_Y: \Omega^\bullet(M)\to \Omega^{\bullet-k}(M)$, and the Lie derivative $\Lie_Y: \Omega^\bullet(M) \to \Omega^{\bullet-k+1}(M)$.
The precise conventions and basic facts about these operations can be found in Appendix \ref{sec:Cartan}. Let us just note the crucial relation
$$ \Lie_Y = [\iota_Y,d]=\iota_Y \circ d - (-1)^{k} d \circ \iota_Y, $$
known as {\em Cartan's magic formula}.

\begin{definition}\label{definition: Koszul bracket of bivector field}
Let $Z$ be a bivector field on $M$. The {\em Koszul bracket} associated to $Z$ is the operation
\begin{eqnarray*}
 & [\cdot,\cdot]_Z: \Omega^r(M) \times \Omega^s(M) \to \Omega^{r+s-1}(M) &\\
 &\, [\alpha,\beta]_Z:=(-1)^{|\alpha|+1} {\Big{(}}\Lie_Z(\alpha\wedge \beta)-\Lie_Z(\alpha)\wedge \beta - (-1)^{|\alpha|}\alpha \wedge \Lie_Z(\beta){\Big{)}}.&
 \end{eqnarray*}
\end{definition}

When applied to $1$-forms $\alpha,\beta$, the Koszul bracket can be written as
$[\alpha,\beta]_Z=\Lie_{Z^{\sharp}\alpha} \beta -\Lie_{Z^{\sharp}\beta} \alpha-d\langle Z, \alpha\wedge \beta\rangle$ and agrees with formula \eqref{eq:KoszulBracket1f} on page \pageref{eq:KoszulBracket1f}.

\begin{remark}\label{rem:Kos2der}
The Koszul bracket $[\cdot,\cdot]_Z$ satisfies the following identities for all homogeneous $\alpha,\beta,\gamma \in \Omega(M)$:
\begin{itemize}
\item Graded skew-symmetry, i.e.~we have
$$[\alpha,\beta]_Z = -(-1)^{(|\alpha|-1)(|\beta|-1)}[\beta,\alpha]_Z.$$
\item Leibniz rule for $d$, i.e.
$$ d([\alpha,\beta]_Z)=[d\alpha,\beta]_Z + (-1)^{|\alpha|-1} [\alpha,d\beta]_Z.$$
\item Derivation property with respect to the wedge product $\wedge$:
$$ [\alpha,\beta\wedge \gamma]_Z = [\alpha,\beta]_Z\wedge \gamma + (-1)^{(|\alpha|-1)|\beta|}\beta \wedge [\alpha,\gamma]_Z.$$
\end{itemize}
After a shift in degree, i.e.~when working on the graded vector space $\Omega(M)[1]$ defined by $(\Omega(M)[1])^r=\Omega^{r+1}(M)$, the Koszul bracket therefore satisfies most of the identities required from a differential graded Lie algebra.
However, unless we assume that $Z$ is Poisson, i.e.~that it commutes with itself under the Schouten-Nijenhuis bracket, the Koszul bracket will fail to satisfy the graded version of the Jacobi identity.
\end{remark}

The failure of $[\cdot,\cdot]_Z$ to satisfy the Jacobi identity is mild. In fact, it can be encoded by a certain trilinear operator on differential forms. As a preparation, we introduce some notation: for a differential form $\alpha \in \Omega^r(M)$, we have 
$$\alpha^\sharp: TM \to \wedge^{r-1}T^*M, \quad v \mapsto \iota_v\alpha,$$
and, following \cite[\S 2.3]{FZgeo}, we extend this definition to a collection of forms $\alpha_1,\dots,\alpha_n$ by setting
\begin{eqnarray*}
 \alpha_1^\sharp \wedge \cdots \wedge \alpha_n^\sharp: \wedge^n TM &\to& \wedge^{|\alpha_1|+\cdots +|\alpha_n|-n}T^*M \\
 v_1\wedge \cdots \wedge v_n & \mapsto & \sum_{\sigma \in S_n}(-1)^{|\sigma|}\alpha_1^\sharp(v_{\sigma(1)})\wedge \cdots \wedge \alpha_n^\sharp(v_{\sigma(n)}).
\end{eqnarray*}

\begin{definition}\label{definition: trinary bracket}
We define the trinary bracket $[\cdot,\cdot,\cdot]_Z: \Omega^r(M)\times \Omega^s(M)\times \Omega^k(M)\to \Omega^{r+s+k-3}(M)$ associated to the bivector field $Z$
to be
$$[\alpha,\beta,\gamma]_Z:= (\alpha^\sharp \wedge \beta^\sharp \wedge \gamma^\sharp)(\frac{1}{2}[Z,Z]),$$
\end{definition}

\begin{remark}\label{rem:Kos3der}
The trinary bracket $[\cdot,\cdot,\cdot]_Z$ is a derivation in each argument, in the sense that
$$[\alpha, \beta, (\gamma\wedge \tilde{\gamma})]_Z=
[\alpha, \beta, \gamma]_Z\wedge \tilde{\gamma}+(-1)^{(|\alpha|+|\beta|-1)|\gamma|}\gamma\wedge [\alpha,\beta, \tilde{\gamma}]_Z.$$
{The trinary bracket satisfies $$[\alpha, \beta, \gamma]_Z=
-(-1)^{(|\alpha|-1)(|\beta|-1)}
[\beta, \alpha,\gamma]_Z,$$
and similarly when switching the order of $\beta$ and $\gamma$.}

The precise compatibility between $d$, $[\cdot,\cdot]_Z$ and $[\cdot,\cdot,\cdot]_Z$ is the following, and will be proven in Section \ref{subsection: proofs L-infty} below:
\end{remark}

\begin{proposition}\label{proposition: extended Koszul}
Let $Z$ be a bivector field on $M$. The multilinear maps $\lambda_1,\lambda_2,\lambda_3$ on the graded vector space $\Omega(M)[2]$ given by
\begin{enumerate}
\item $\lambda_1$ is the de~Rham differential $d$, acting on $\Omega(M)[2]$,
\item $\lambda_2(\alpha[2] \odot \beta[2])=-\Big(\Lie_Z(\alpha\wedge \beta) - \Lie_Z(\alpha)\wedge \beta - (-1)^{\vert \alpha \vert}\alpha \wedge \Lie_Z(\beta)\Big)[2] = (-1)^{|\alpha|} ([\alpha,\beta]_Z)[2]$,
\item and
\begin{eqnarray*}
\lambda_3(\alpha[2]\odot \beta[2]\odot \gamma[2])=(-1)^{|\beta|+1}\Big(\alpha^\sharp\wedge \beta^\sharp \wedge \gamma^\sharp(\frac{1}{2}[Z,Z])\Big)[2]
\end{eqnarray*}
\end{enumerate}
define the structure of an $L_\infty[1]$-algebra
on $\Omega(M)[2]$, or -- equivalently -- the structure of an $L_\infty$-algebra on $\Omega(M)[1]$, see Appendix \ref{appendix: reminder L-infty}.
\end{proposition}

In a nutshell, Proposition \ref{proposition: extended Koszul}
asserts that $\lambda_1$, $\lambda_2$ and $\lambda_3$ obey a family of quadratic relations.
Besides the relations which involve only the de~Rham differential $d$ and the Koszul bracket $[\cdot,\cdot]_Z$, there is a relation which asserts that the Jacobiator of $[\cdot,\cdot]_Z$, seen as a chain map from $\Omega(M)[2]\odot \Omega(M)[2]\odot \Omega(M)[2]$ to $\Omega(M)[2]$, is zero-homotopic, with homotopy provided by $\lambda_3$.
A concise summary on $L_\infty[1]$-algebras can be found in Appendix \ref{appendix: reminder L-infty}.

\begin{proposition}\label{proposition: extended Koszul 2}
Let $Z$ be a bivector field on $M$.
There is an $L_\infty[1]$-{\em isomorphism} $\psi$ from the
$L_\infty[1]$-algebra $(\Omega(M)[2],\lambda_1,\lambda_2,\lambda_3)$
of Proposition \ref{proposition: extended Koszul}
to the $L_\infty[1]$-algebra $(\Omega(M)[2],\lambda_1)$, i.e.
the (twice suspended) de~Rham complex of $M$.
\end{proposition}

The interested reader can find the proof in Section \ref{subsection: proofs L-infty} below. 
Let us just point out that we have a rather explicit knowledge of the $L_{\infty}[1]$-isomorphism $\psi$
between $(\Omega(M)[2],\lambda_1,\lambda_2,\lambda_3)$
and $(\Omega(M)[2],\lambda_1)$.
We will make use of this knowledge below.

\begin{remark}
In case we can find an involutive complement $G$ to $K\subset TM$,
 the construction from 
Proposition \ref{proposition: extended Koszul}
yields the Koszul dg Lie algebra
associated to the regular Poisson structure $Z$ on $M$ given by
the negative inverse of $\eta\vert_{\wedge^2 G}$.
For $Z$ a Poisson bivector, Proposition \ref{proposition: extended Koszul 2} was established by Fiorenza and Manetti in \cite{Fiorenza-Manetti}.
\end{remark}

We now turn to the geometry encoded by the
\lione
$(\Omega(M)[2],\lambda_1,\lambda_2,\lambda_3)$.
To this end, recall that we can naturally associate the following equation to such a structure:

\begin{definition}\label{definition: Maurer-Cartan equation}
An element $\beta \in \Omega^2(M)$ is a {\em Maurer-Cartan element} 
of $(\Omega(M)[2],\lambda_1,\lambda_2,\lambda_3)$ if it satisfies the {\em Maurer-Cartan equation}
$$ d(\beta[2]) + \frac{1}{2}\lambda_2(\beta[2]\odot \beta[2]) + \frac{1}{6}\lambda_3(\beta[2]\odot \beta[2]\odot \beta[2])=0.$$
We denote the set of Maurer-Cartan elements by $\MC(Z)$.
\end{definition}

Recall that in Equation \eqref{eq:alphabeta} we introduced a map
 $F \colon \mathcal{I}_Z\to \wedge^2 T^*M$, 
where  $\mathcal{I}_Z \subset \Omega^2(M)$  consists of those 2-forms $\beta$
for which $\mathrm{id} + Z^\sharp \beta^\sharp$ is invertible. 

\begin{corollary}\label{corollary: MC for bivector fields}
There is an open subset $\mathcal{U}\subset \mathcal{I}_Z$, which contains the zero section of $\wedge^2 T^*M$, such that
a $2$-form $\beta \in \Gamma(\mathcal{U})$ is a Maurer-Cartan element of $(\Omega(M)[2],\lambda_1,\lambda_2,\lambda_3)$ if, and only if, the $2$-form $F(\beta)$
is closed.
\end{corollary}

A proof of Corollary \ref{corollary: MC for bivector fields} is given after Proposition \ref{proposition: algebraic map = geometric map} below.
{It can be shown {\cite{SZDirac}} that one} can in fact take $\mathcal{U}=\mathcal{I}_Z$. To lighten the notation, in the rest of the paper we will use this fact. \\

Summing up our findings, we have a smooth, fibrewise mapping 
$$ F: \mathcal{I}_Z \to \wedge^2 T^*M,$$ which has the intriguing property of mapping $2$-forms, which are Maurer-Cartan elements of $(\Omega(M)[2],\lambda_1,\lambda_2,\lambda_3)$, to {\em closed} $2$-forms.

We observe that the $L_\infty[1]$-isomorphism from Proposition \ref{proposition: extended Koszul 2}
$$ \psi: (\Omega(M)[2],\lambda_1,\lambda_2,\lambda_3)\to (\Omega(M)[2],\lambda_1)$$
also induces -- modulo convergence issues -- a map
$$ \psi_*: \Omega^2(M) \to \Omega^2(M), \quad \beta \mapsto \psi_*(\beta):= \sum_{k \ge 1} \frac{1}{k!}\psi_k(\beta^{\odot k}),$$
which has the property that it sends Maurer-Cartan elements
of $(\Omega(M)[2],\lambda_1,\lambda_2,\lambda_3)$ to {\em closed} $2$-forms.
The following result asserts that for $\beta$ sufficiently small, the convergence of $\psi_*(\beta)$ is guaranteed, in which case we recover $F(\beta)$ in the limit.

\begin{proposition}\label{proposition: algebraic map = geometric map}
For $\beta$ a $2$-form which is sufficiently $\mathcal{C}^0$-small, the power series
$ \psi_*(\beta)$
converges in the $\mathcal{C}^\infty$-topology to $F(\beta)$.
\end{proposition}

As before, we postpone the proof of this result to Subsection \ref{subsection: proofs L-infty}. Let us demonstrate that {Proposition
\ref{proposition: algebraic map = geometric map}} yields a proof of Corollary \ref{corollary: MC for bivector fields}:
Suppose $\beta$ is a $2$-form which is sufficiently $\mathcal{C}^0$-small for Proposition \ref{proposition: algebraic map = geometric map} to apply.
We compute
\begin{eqnarray*}
dF(\beta) = d \psi_*(\beta) = \sum_{k\ge 1}\frac{1}{k!} d\psi_k(\beta^{\odot k}) = \sum_{k\ge 1} \frac{1}{k!}\psi_{k+1}(\beta^{\odot k}\odot (\lambda_1(\beta) + \frac{1}{2}\lambda_2(\beta\odot \beta) + \frac{1}{6}\lambda_3(\beta\odot \beta\odot \beta)).
\end{eqnarray*}
This shows that if $\beta$ satisfies the Maurer-Cartan equation, $F(\beta)$ is closed. To obtain the reverse implication,
one repeats the same line of arguments to the inverse of $\psi$.
 \\
 
\subsection{Proofs for Section \ref{subsection: Koszul for bivectors}.}
\label{subsection: proofs L-infty}

In this section we provide proofs of the statements from 
Section \ref{subsection: Koszul for bivectors} which involve $L_\infty$-algebras.
We refer to Appendix \ref{appendix: reminder L-infty} for background material.

\subsubsection{\underline{The proofs of Proposition \ref{proposition: extended Koszul} and Proposition \ref{proposition: extended Koszul 2}}}
\label{subsubsection: proofs of extended Koszul 1 and 2}
 
The proofs in this section rely on the calculus of differential operators on graded commutative algebras, and their associated Koszul brackets, see Appendix \ref{appendix: reminder L-infty}.
For us, the following example of differential operators is of central importance:

\begin{example}
Let $Y$ be a $k$-multivector field on $M$, i.e.~a section of $\Gamma(\wedge^k TM)$.
The insertion operator
$$ \iota_Y: \Omega^\bullet(M) \to \Omega^{\bullet-k}(M)$$
is a differential operator of order $\le k$ on $\Omega(M)$.
Since graded derivations are differential operators of order $\le 1$,
and since $[\Diffop_k(A),\Diffop_l(A)]\subset \Diffop_{k+l-1}(A)$, see Appendix \ref{appendix: BV},
we find that the Lie derivative
$$\Lie_Y = [\iota_Y,d]$$
is also a differential operator of order $\le k$ on $\Omega(M)$.
\end{example}

Let $A$ be a graded commutative, unital dg algebra.
Given an endomorphism $f$ of $A$, the {\em Koszul brackets} of $f$ are 
a sequence of multilinear operators
$\Koszul(f)_n: \odot^n A\to A$ defined iteratively by
$\Koszul(f)_1=f$ and
\begin{eqnarray*}
\Koszul(f)_n(a_1\odot \cdots \odot a_n) &=& +\Koszul(f)_{n-1}(a_1\odot \cdots \odot a_{n-2}\odot a_{n-1}a_n)\\
&& - \Koszul(f)_{n-1}(a_1\odot \cdots \odot a_{n-1})a_n\\
&&-(-1)^{|a_{n-1}||a_n|}\Koszul(f)_{n-1}(a_1\odot \cdots \odot a_{n-1}\odot a_n)a_{n-1},
\end{eqnarray*}
Using the natural identification $\Hom(\odot A,A)\cong \Coder(\odot A)$, this construction yields a morphism of dg Lie algebras
\begin{equation}\label{eq:cK}
 \Koszul: \mathrm{End}_*(A) \to \Coder(\odot A)
\end{equation}
from the dg Lie algebra of endomorphisms of $A$ which annihilate $1_A$, to the dg Lie algebra of coderivations
of the (reduced) symmetric coalgebra on $A$ (with the commutator bracket),
c.f.~\cite{Voronov1}.

\begin{lemma}\label{lemma: Koszul brackets for contraction}
Let $V$ be a finite-dimensional vector space.
Given $Y_1,\dots,Y_n \in V$, we consider $Y:=Y_1\wedge \cdots \wedge Y_n \in \wedge^n V$, and the corresponding differential operator
$$ \iota_Y :\wedge^\bullet V^* \mapsto \wedge^{\bullet-n} V^*$$
on the commutative graded algebra $\wedge V^*$.
For $r\le n$, the Koszul brackets of $\iota_Y$ are given by
$$ \Koszul(\iota_Y)_r(a_1\odot \cdots \odot a_r)=\sum_{i_1+\cdots +i_r=n} \sum_{\sigma \in S(i_1,\dots,i_r)}(-1)^\sharp (\iota_{Y_{\sigma(1)}}\cdots \iota_{Y_{\sigma(i_1)}}a_1)\cdots (\cdots \iota_{Y_{\sigma(n)}}a_r),$$
where we sum over all tuples $(i_1,\dots,i_r) \in \mathbb{Z}^r$ with
$i_1+\cdots+ i_r=n$ such that all $i_r \ge 1$,
$S(i_1,\dots,i_r)$ is the set of all $(i_1,\dots,i_r)$-unshuffles, and the sign is
$$ \sharp = |\sigma| + \sum_{p=2}^r i_p(|a_1|+\cdots + |a_{p-1}|).$$
\end{lemma}

\begin{proof}
The proof proceeds by induction on $r$. 
For $r=1$, the claimed formula reads
$$\Koszul(\iota_Y)_1(a)=\iota_{Y_1}\cdots \iota_{Y_n} a,$$
which -- by our conventions, see Appendix \ref{appendix: reminder L-infty} -- indeed equals $\iota_Y a$.

We now assume that we verified the formula for all $r\ge 1$.
Our task is to show that for $r+1$, the right-hand side 
of the claimed equality satisfies the same recursion as
$\Koszul(\iota_Y)_{r+1}$.
This is a straightforward computation.
\end{proof}

\begin{remark}\label{remark: Koszul}
We will be particularly interested in the special case $n=r$, where the formula specializes to
$$\Koszul(\iota_Y)_n(a_1\odot \cdots \odot a_n):= (-1)^{(n-1)|a_1|+(n-2)|a_2|+\cdots + |a_{n-1}|} (a_1^\sharp \wedge \cdots \wedge a_n^\sharp)(Y),$$
using the notation of Subsection \ref{subsection: Koszul for bivectors}.
\end{remark}

We now provide a proof of Proposition \ref{proposition: extended Koszul}, relying on a
general construction from \cite{Kravchenko,Voronov1}.
Let $Z$ be an arbitrary bivector field on $M$. Cartan calculus -- which we recap in Appendix \ref{sec:Cartan} -- implies that the operator
$$\Delta_Z:= d - t \Lie_Z - t^2 \frac{1}{2} \iota_{[Z,Z]}$$
on $\Omega(M)[[t]]$, $t$ a formal variable of degree $2$, squares to zero. In fact, $\Delta_Z$ equips $\Omega(M)$ with the structure of a commutative $BV_\infty$-algebra of degree $1$ in the sense of \cite{Kravchenko}, {which amounts to the fact that $d$ is a derivation, $\Lie_Z$ is a differential operator of order $\le 2$,
and $\iota_{[Z,Z]}$ is a differential operator of order $\le 3$.}
This $BV_\infty$-algebra structure on $\Omega(M)$ was considered by Fiorenza-Manetti
in the case of $Z$ being a Poisson bivector field, cf.~\cite{Fiorenza-Manetti},
and by Dotsenko-Shadrin-Valette in the case of $Z$ being a Jacobi bivector field, cf.~\cite{DSV}.

\begin{proof}[Proof of Proposition \ref{proposition: extended Koszul}]
As noted above, the operator
$$ \Delta_Z=d - t\Lie_Z - t^2\frac{1}{2}\iota_{[Z,Z]}$$
equips $\Omega(M)$ with the structure of a commutative $BV_\infty$-algebra.
Proposition \ref{proposition: Koszul brackets} then asserts that the sequence of operations
$$(d,\Koszul(-\Lie_Z)_2,\Koszul(-\frac{1}{2}\iota_{[Z,Z]})_3)$$
equips $\Omega(M)[2]$ with the structure of an $L_\infty[1]$-algebra.
The structure maps $\lambda_1$ and $\lambda_2$ are easy to match with $d$ and $\Koszul(\Lie_Z)_2$, respectively.
Concerning $\lambda_3$, we know by Remark \ref{remark: Koszul} that
$$ \Koszul(-\frac{1}{2}\iota_{[Z,Z]})_3(\alpha\odot \beta\odot\gamma)= (-1)^{|\beta|+1}(\alpha^\sharp \wedge \beta^\sharp \wedge \gamma^\sharp)(\frac{1}{2}[Z,Z]),$$
which concludes the proof.
\end{proof}

We next turn to Proposition \ref{proposition: extended Koszul 2}.
We consider now the second Koszul bracket associated to the contraction by the bivector field $Z$,
$$\Koszul(\iota_Z)_2: \Omega(M)^{\odot 2} \to \Omega(M)$$
 One can extend $\Koszul(\iota_Z)_2$ in a unique way to a coderivation $R_Z$ of $\bigodot(\Omega(M)[2])$, and as such it has degree zero.
Since $R_Z$ acts in a pro-nilpotent manner on $\bigodot(\Omega(M)[2])$, it integrates to an automorphism $e^{R_Z}$ of the graded coaugmented coalgebra $\bigodot(\Omega(M)[2])$.

We thank Ruggero Bandiera for helpful conversations which led to the following proof.

\begin{proof}[Proof of Proposition \ref{proposition: extended Koszul 2}]
We claim that $\psi:=e^{R_Z}$ defines an $L_\infty[1]$-isomorphism
$$ \psi: (\Omega(M)[2],\lambda_1,\lambda_2,\lambda_3) \to (\Omega(M)[2],\lambda_1),$$
with inverse given by $e^{-R_Z}$.
Equivalently, we can verify that the coderivation
$$ e^{-R_Z}\circ \widehat{\lambda}_1 \circ e^{R_Z}$$
corresponds to $(\lambda_1,\lambda_2,\lambda_3)$, where  $\widehat{\lambda}_1$ is the coderivation of $\bigodot(\Omega(M)[2])$ extending $\lambda_1$.
Notice that
$$ e^{-R_{Z}}\circ \widehat{\lambda}_1 \circ e^{R_{Z}} = \widehat{\lambda}_1 - [R_{Z},\widehat{\lambda}_1] + \frac{1}{2}[R_{Z},[R_{Z},\widehat{\lambda}_1]] - \frac{1}{3!}[R_{Z},[R_{Z},[R_{Z},\widehat{\lambda}_1]]]]+\cdots,$$
where $[\cdot,\cdot]$ denotes the commutator bracket, 
i.e.~the $(k+1)$-th structure map of $ e^{-R_Z}\circ \widehat{\lambda}_1 \circ e^{R_Z}$ can be read off from
$$ \frac{1}{k!} (-1)^k \mathrm{ad}(R_{Z})^k\widehat{\lambda}_1.$$

We compute the element corresponding to $ - [R_{Z},\widehat{\lambda}_1] $ under the natural identification $\Hom(\odot A,A)\cong \Coder(\odot A)$. 
Using the fact that the map \eqref{eq:cK} respects commutator brackets, and that $d$ is a derivation, while $\iota_Z$ is differential operators of order $\le 2$, we obtain
\begin{eqnarray*}
-[\Koszul(\iota_Z)_2,\Koszul(d)_1]=-\Koszul([\iota_Z,d])_2=-\Koszul(\Lie_Z)_2=\lambda_2.
\end{eqnarray*}
In the same manner -- this time also using that $\Lie_Z$ is a differential operator of order $\le 2$ -- we find
\begin{eqnarray*}
\frac{1}{2}[R_{Z},[R_{Z},\widehat{\lambda}_1]]=\frac{1}{2}[\Koszul(\iota_Z)_2,\Koszul(\Lie_Z)_2]=\frac{1}{2}\Koszul([\iota_Z,\Lie_Z])_3=-\frac{1}{2}\Koszul(\iota_{[Z,Z]})_3=\lambda_3.
\end{eqnarray*}
Finally, we find
\begin{eqnarray*}
\frac{1}{6}[\Koszul(\iota_Z)_2,-\frac{1}{2}\Koszul(\iota_{[Z,Z]})_3]=
\frac{1}{12}\Koszul([\iota_Z,-\iota_{[Z,Z]}])_4
=0,
\end{eqnarray*}
where we made use of the fact that contraction operators 
commute in the graded sense.
\end{proof}

\begin{remark}
{An alternative proof, pointed out to us by V.~Dotsenko, proceeds by noticing that the operator
$$ \Delta_Z=d - t\Lie_Z - t^2\frac{1}{2}\iota_{[Z,Z]}$$
is conjugate to $d$ via the automorphism of $\Omega(M)[[t]]$
generated by $t\iota_Z$. In the case of $Z$ being Poisson,
this was previously observed by Fiorenza-Manetti in \cite{Fiorenza-Manetti},
and, in the case of $Z$ being Jacobi, by Dotsenko-Shadrin-Valette in \cite{DSV}.}

{
Combining this observation with the fact that the higher Koszul brackets 
are compatible with the commutator brackets, see \eqref{eq:cK},
yields Proposition \ref{proposition: extended Koszul 2}.}
\end{remark}

\subsubsection{\underline{The proof of Proposition \ref{proposition: algebraic map = geometric map}}}
\label{subsubsection: proof of algebraic = geometric}

As we just saw, the $L_\infty[1]$-isomorphism
$\psi: (\Omega(M)[2],\lambda_1,\lambda_2,\lambda_3)\to (\Omega(M)[2],\lambda_1)$
is given by $e^{R_Z}$, with the coderivation $R_Z$ determined by
$$ R_Z(\alpha[2]\odot\beta[2]) = \Big(\iota_Z(\alpha\wedge \beta) - \iota_Z(\alpha)\wedge \beta - \alpha\wedge \iota_Z(\beta)\Big)[2].$$
Notice that this equals the second Koszul bracket $\Koszul_2(\iota_Z)$ of $\iota_Z$, see Section \ref{subsubsection: proofs of extended Koszul 1 and 2}.

We are interested in $e^{R_Z}(e^{\beta[2]})$ for $\beta\in \Omega^2(M)$. 
Since $e^{R_Z}$ is an automorphism of the coalgebra $\bigodot(\Omega(M)[2])$
it maps $e^{\beta[2]}$ to an element of the form $e^{\alpha[2]}$. Our aim is to derive a formula for $\alpha$.

As a preparation, we prove the following

\begin{lemma}\label{lemma: sharp and the coderivation}
For $\beta$, $\tilde{\beta}\in \Omega^2(M)$, the $2$-form
$ \gamma[2]:=R_Z(\beta[2]\odot \tilde{\beta}[2])$
is determined by $$\gamma^\sharp = -\left(\beta^\sharp Z^\sharp \tilde{\beta}^\sharp + \tilde{\beta}^\sharp Z^\sharp \beta^\sharp\right).$$
\end{lemma}

\begin{proof}
Let us fix two $2$-forms $\beta$ and $\tilde{\beta}$.
 We assume without loss of generality that $Z=Z_1\wedge Z_2$
 and we know that $R_Z(\beta[2]\odot \tilde{\beta}[2])=\Koszul_2(\iota_Z)(\beta\odot \tilde{\beta})[2]$. 
By Lemma \ref{lemma: Koszul brackets for contraction}, we have
$$ \gamma= \Koszul(\iota_Z)_2(\beta\odot \tilde{\beta})= \iota_{Z_1}\beta \wedge \iota_{Z_2}\tilde{\beta} - \iota_{Z_2}\beta \wedge \iota_{Z_1}\tilde{\beta}$$
and therefore
$$ \gamma(v_1,v_2)=\beta(Z_1,v_1)\tilde{\beta}(Z_2,v_2) -
\beta(Z_1,v_2)\tilde{\beta}(Z_2,v_1) -
\beta(Z_2,v_1)\tilde{\beta}(Z_1,v_2) +
\beta(Z_2,v_2)\tilde{\beta}(Z_1,v_1).$$
On the other hand, if we evaluate the bilinear skew-symmetric form corresponding to the map $-\beta^\sharp Z^\sharp \tilde{\beta}^\sharp$ on two vectors $v_1$, $v_2 \in V$, we find
$$
-\langle \beta^\sharp Z^\sharp \tilde{\beta}^\sharp(v_1),v_2\rangle = \beta(Z_2,v_2)\tilde{\beta}(Z_1,v_1) - \beta(Z_1,v_2)\tilde{\beta}(Z_2,v_1),$$
and similarly when one switches $\beta$ and $\tilde{\beta}$.
This finally yields
$$\gamma(v_1,v_2) = \langle - \left(\beta^\sharp Z^\sharp \tilde{\beta}^\sharp + \tilde{\beta}^\sharp Z^\sharp \beta^\sharp \right)(v_1),v_2 \rangle,$$
which was our claim.
\end{proof}

We introduce a formal parameter $t$ and define $\alpha(t) \in \Omega^2(M)[[t]]$
by
\begin{equation}\label{equation: ODE} e^{tR_Z}e^{\beta}=:e^{\alpha(t)}.
\end{equation}
We write $\alpha(t)=\sum_{j=0}^\infty \alpha_j t^j$.

\begin{lemma}\label{lemma: sharp and the coderivation 2}
The coefficients $\alpha_j\in \Omega^2(M)$ of $\alpha(t)$ are determined by
$\alpha_j^\sharp = (-1)^j\beta^\sharp (Z^\sharp \beta^\sharp)^{j}.$ 
\end{lemma}

\begin{proof}
The statement is obviously true for $j=0$.
Now suppose we proved the statement for all $i<j$ already.
Differentiating both sides of Equation \eqref{equation: ODE} with respect to $t$ leads to
$$ R_Z(e^{tR_Z}e^{\beta}) = R_Z(e^{\alpha(t)})=\dot{\alpha}(t)\odot e^{\alpha(t)}.$$
 After applying the projection $\bigodot \Omega(M)[2]\to \Omega(M)[2]$,
one obtains $R_Z({\frac{1}{2}\alpha(t)}\odot {\alpha(t)})=\dot{\alpha}(t)$.  Comparing the coefficients of $t^{j-1}$,
we find the equation
$$ \alpha_{j} = \frac{1}{j}R_Z\Big(\frac{1}{2} \sum_{r+s=j-1} \alpha_r \odot \alpha_s\Big).$$
Applying induction and Lemma \ref{lemma: sharp and the coderivation},
this yields
\begin{eqnarray*}
\alpha_{j}^\sharp &=& \frac{1}{j}(-1)^j\sum_{0\le i\le j-1}\beta^\sharp(Z^\sharp\beta^\sharp)^{i}Z^\sharp\beta^\sharp(Z^\sharp\beta^\sharp)^{j-i-1} \\
&=& (-1)^j \beta^\sharp(Z^\sharp\beta^\sharp)^j,
\end{eqnarray*}
which finishes the proof.
\end{proof}

\begin{proof}[Proof of Proposition \ref{proposition: algebraic map = geometric map}]
Setting $t=1$ in Equation \eqref{equation: ODE} and applying the projection $\bigodot \Omega(M)[2]\to \Omega(M)[2]$ we find
 $\alpha=\alpha(1)=\sum_{j\ge 0} \alpha_j$.
By Lemma \ref{lemma: sharp and the coderivation 2}, we know that
$$(\psi_*(\beta))^\sharp =\alpha^\sharp= \sum_{j\ge 0}\alpha_j=\sum_{j\ge 0}(-1)^j\beta^\sharp(Z^\sharp\beta^\sharp)^j.$$
For $\beta$ sufficiently small with respect to the $\mathcal{C}^0$-topology,
this formal series converges uniformly with respect to the $\mathcal{C}^\infty$-topology to $\beta^\sharp(\mathrm{id}+Z^\sharp\beta^\sharp)^{-1}$.
 The latter expression coincides with $(F (\beta))^\sharp$, see Equation \eqref{eq:alphabeta}.

\end{proof}

\subsection{The Koszul $L_\infty$-algebra of a pre-symplectic manifold}
\label{subsection: Koszul for pre-symplectic}

We return to the pre-symplectic setting, i.e.~suppose $\eta$ is a pre-symplectic structure on $M$.
Let us fix a complementary subbundle $G$ to the kernel $K\subset TM$
of $\eta$
and let $Z$ be the bivector field on $M$ determined by
$ Z^\sharp = -(\eta\vert_G^\sharp)^{-1}$.

\begin{theorem}\label{theorem: construction - Koszul L-infty}
The $L_{\infty}[1]$-algebra structure  
 on $\Omega(M)[2]$ associated to the bivector field $Z$, see Proposition \ref{proposition: extended Koszul},
maps $\Omega_\hor(M)[2]$ to itself.
The subcomplex $\Omega_\hor(M)[2] \subset \Omega(M)[2]$ therefore inherits the structure of an
$L_\infty[1]$-algebra, which we call the {\em Koszul $L_\infty[1]$-algebra} of $(M,\eta)$. Moreover, we call the corresponding $L_\infty$-algebra structure on $\Omega_\hor(M)[1]$, see Appendix \ref{appendix: reminder L-infty}, the {\em Koszul $L_\infty$-algebra} of $(M,\eta)$.
\end{theorem}

\begin{proof}

We have to show that $\Omega_\hor(M)[2]$ is preserved by the operations $\lambda_1$, $\lambda_2$ and $\lambda_3$
as defined in Proposition \ref{proposition: extended Koszul}.
Notice that these structure maps  operate in a local manner on $\Omega(M)$. 

We already know that $\Omega_\hor(M)$ is a subcomplex of $\Omega(M)$, so the claim is clear for $\lambda_1$.

For the binary map $\lambda_2$, we notice that it suffices to show that
$\Gamma(K^\circ)$ is closed under the Koszul bracket; the reasons being that $\Omega_\hor(M)$
is the ideal generated by $\Gamma(K^\circ)$ and that the Koszul bracket is a derivation in each argument. We already showed in Lemma \ref{lem:g*g*g*}, Section \ref{section: bivec}, that $\Gamma(K^\circ)$ is closed under the Koszul bracket.

To see that $\lambda_3$ maps $\Omega_\hor(M)[2]^{\odot 3}$ to $\Omega_\hor(M)[2]$, recall that
$\Omega_\hor(M)=\Gamma(\wedge^{\ge 1}G^*\otimes \wedge K^*)$ and that the component of $[Z,Z]$ in $\Gamma(\wedge^3 G)$ vanishes by Lemma \ref{lem:noggg} in Section \ref{section: bivec}.
\end{proof}

\begin{definition}\label{definition: MC for pre-symplectic}
We denote by $\MC(\eta)$ the set of Maurer-Cartan elements of
the Koszul \lione of $(M,\eta)$.
\end{definition}

Recall from Section \ref{subsection: Dirac parametrization}
that $\mathcal{I}_Z$ denotes the subset of those elements $\beta$ of
$\wedge^2 T^*M$, for which $\mathrm{id}+Z^\sharp\beta^{\sharp}$ is invertible.
Recall further that we constructed a map
\begin{align}\label{exp-recalled} \exp_\eta: \mathcal{I}_Z\to \wedge^2 T^*M, \quad \beta \mapsto \eta + F(\beta),
\end{align}
which is fibre-preserving, and maps $\mathcal{I}_Z$ diffeomorphically onto $\{\eta'\in \Omega^2(M)|\;\eta'-\eta\in \mathcal{I}_{-Z}\}$,
sending the zero section to $\eta$. Finally, we saw that this map restricts to a  diffeomorphism from
$\mathcal{I}_Z \cap ((K^*\otimes G^*)\oplus \wedge^2 G^*)$ to the $2$-forms 
$\eta'$ of rank $k$ such that $\eta'-\eta\in \mathcal{I}_{-Z}$.

Drawing from the results established up to this point -- {namely Theorem \ref{theorem: almost Dirac structures}  and Corollary \ref{corollary: MC for bivector fields}} -- we are now ready to prove our main result:

\begin{theorem}\label{theorem: main result}
Let $(M,\eta)$ be a pre-symplectic manifold. The choice of a complement $G$ to the kernel of $\eta$ determines a bivector field $Z$ by requiring $Z^\sharp=-(\eta^\sharp \vert_G)^{-1}$. 
Suppose $\beta$ is a $2$-form on $M$, which lies in $\mathcal{I}_Z$.
The following statements are equivalent:

\begin{enumerate}
\item $\beta$ is a Maurer-Cartan element of the Koszul \lione $\Omega_\hor(M)[2]$ of $(M,\eta)$, which was introduced in Theorem \ref{theorem: construction - Koszul L-infty}.
\item The image of $\beta$ under the map $\exp_\eta$, which is recalled in \eqref{exp-recalled}, is a pre-symplectic structure
of the same rank as $\eta$.
\end{enumerate}
\end{theorem}

\begin{proof}
We already know that $\beta$ being horizontal is equivalent to $\exp_\eta(\beta)$ being of the same rank as $\eta$, see Theorem \ref{theorem: almost Dirac structures} (iii) in Section \ref{subsection: Dirac parametrization}.
Clearly  $\beta$ is Maurer-Cartan in $\Omega_\hor(M)[2]$
if, and only if, it is Maurer-Cartan in $\Omega(M)[2]$.
By Corollary \ref{corollary: MC for bivector fields}, the latter is equivalent to $F(\beta)$ being closed. In turn this is equivalent to  $\exp_\eta(\beta)$ being closed, since these two forms differ by the closed $2$-form $\eta$.\end{proof}

Rephrasing the above result, the fibrewise map
$$ \exp_\eta: \mathcal{I}_Z \cap ((K^*\otimes G^*)\oplus \wedge^2 G^*) \to (\wedge^2 T^*M)_k$$
restricts, on the level of sections, to an injective map
$$\boxed{ \exp_\eta: \Gamma(\mathcal{I}_Z)\cap \MC(\eta) \to  \Presym^k(M)}$$
with image an open neighborhood of $\eta$ in 
$\Presym^k(M)$ (equipped with the $\mathcal{C}^0$-topology). That is, for $\beta\in \Omega^2_\hor(M)$ sufficiently $\mathcal{C}^0$-small,
being Maurer-Cartan is equivalent to $\exp_\eta(\beta)$
being a pre-symplectic structure of rank equal to the rank of $\eta$.

\subsubsection {\underline{Quotient $L_{\infty}[1]$-algebras}}

Theorem \ref{theorem: construction - Koszul L-infty} asserts that the multiplicative ideal of horizontal forms $\Omega_\hor(M)$
is closed with respect to the $L_\infty[1]$-algebra structure
maps $\lambda_1,\lambda_2,\lambda_3$ from Proposition \ref{proposition: extended Koszul}. We now refine this result. For $k\ge 0$, we denote by 
$F^{k}\Omega(M)$ the $k$'th power of $\Omega_\hor(M)\subset \Omega(M)$. Given the choice of a subbundle $G\subset TM$ which is a complement to the kernel $K$ of the pre-symplectic form $\eta$, we have the  identification 
$$F^k\Omega(M)=\Gamma(\wedge^{\bullet}K^*\otimes \wedge^{\ge k}G^*).$$
This gives us a filtration
$$F^{0}\Omega(M)=\Omega(M) \, \supset \, F^{1}\Omega(M)=\Omega_\hor(M) \, \supset \, F^{2}\Omega(M)\,  \supset \, \dots \, \supset \, \{0\}.$$

 \begin{lemma}\label{lem:degrees}
The $L_\infty[1]$-algebra structure maps $\lambda_1,\lambda_2,\lambda_3$ on $\Omega(M)[2]$ associated to the bivector field $Z$, see Proposition \ref{proposition: extended Koszul}, satisfy for all $k\ge 0$:
\begin{enumerate}
\item[i)] $\lambda_1(F^{k}\Omega(M)[2])\subset F^{k}\Omega(M)[2],$
\item[ii)] $\lambda_2(F^{k}\Omega(M)[2],F^{l}\Omega(M)[2])\subset F^{k+l-1}\Omega(M)[2]$,
\item[iii)] $\lambda_3(F^{k}\Omega(M)[2],F^{l}\Omega(M)[2],F^{m}\Omega(M)[2])\subset F^{k+l+m-2}\Omega(M)[2]$.
\end{enumerate}
\end{lemma}
\begin{proof}

We use the notation $\Omega^{j,k}:=
\Gamma(\wedge^{j}K^*\otimes \wedge^{k}G^*)$ for brevity
and also suppress the shift in degrees.
\begin{enumerate}
\item[i)]  For $\alpha\in \Gamma(G^*)=\Gamma(K^{\circ})$, the involutivity of $K$ and the usual formula for the de~Rham differential imply that
$d\alpha|_{\wedge^2 K}=0$, i.e.~$d\alpha\in \Omega^{1,1}+\Omega^{0,2}$. Since $d$ is a (degree one) derivation of the wedge product, we obtain $d(\Omega^{0,k})\subset 
\Omega^{1,k}+\Omega^{0,k+1}$.
For $\alpha\in \Gamma(K^*)$ in general we can only state that $d\alpha\in \Omega^{2,0}+\Omega^{1,1}+\Omega^{0,2}$, so that $d(\Omega^{j,0})\subset 
\Omega^{j+1,0}+\Omega^{j,1}+\Omega^{j-1,2}$. The statement follows from this.
\item[ii)]  We have
\begin{itemize}
\item [] $\lambda_2(\Gamma(G^*),\Gamma(G^*))\subset \Gamma(G^*)$
\item [] $\lambda_2(\Gamma(K^*),\Gamma(K^*))\equiv 0$
\item [] $\lambda_2(\Gamma(K^*),\Gamma(G^*)) \subset \Gamma(K^*)+\Gamma(G^*)$
\end{itemize}
The first statement is established in Lemma \ref
{lem:g*g*g*}, Section \ref{section: bivec}.
The second is
 a consequence of the formula \eqref{eq:KoszulBracket1f}, page \pageref{eq:KoszulBracket1f}.
The third statement is clear.
Since $\lambda_2$ is a graded  bi-derivation with respect to the wedge product, item ii) follows.
\item[iii)]   By Lemma \ref{lem:noggg} we know that 
$[Z,Z]$ has
 no component in $\wedge^3 G$. The formula for $\lambda_3$ in Proposition \ref{proposition: extended Koszul} implies the statement.\end{enumerate}
\end{proof}

\begin{remark}
When $G$ is involutive, Lemma \ref{lem:degrees}  can be improved to the following statement:
$\lambda_1(\Omega^{j,k})\subset \Omega^{j+1,k}+\Omega^{j,k+1}$ and 
$\lambda_2(\Omega^{\bullet,k}, \Omega^{\bullet,l})\subset  \Omega^{\bullet,k+l-1}$ (recall that $\lambda_3$ vanishes).
This follows from the proof of Lemma \ref{lem:degrees}, together with the following observations which hold when $G$ is involutive.
First:  $\alpha\in \Gamma(K^*)$ satisfies $d\alpha\in \Omega^{2,0}+\Omega^{1,1}$. Second: using this and formula \eqref{eq:KoszulBracket1f}, page \pageref{eq:KoszulBracket1f}, one has
   $\lambda_2(\Gamma(K^*),\Gamma(G^*)) \subset \Gamma(K^*)$.
\end{remark}
Returning to the general case, Lemma \ref{lem:degrees} allows us to
refine Theorem \ref{theorem: construction - Koszul L-infty}
as follows:
\begin{corollary}\label{cor:liideal}
\hspace{0cm}
\begin{itemize}
\item [1)] $(F^{1}\Omega(M))[2]=\Omega_\hor(M)[2]$ is an $L_{\infty}[1]$-subalgebra of $\Omega(M)[2]$.
\item [2)] $(F^{k}\Omega(M))[2]$ is an $L_{\infty}[1]$-ideal of $\Omega_\hor(M)[2]$ for all $k\ge 2$. Hence we get a sequence of \liones and strict morphisms between them
$$ \left(\faktor{\Omega_{\hor}(M)}{F^2\Omega(M)}\right)[2] \longleftarrow \left(\faktor{\Omega_{\hor}(M)}{F^3\Omega(M)}\right)[2] \longleftarrow \cdots \longleftarrow \Omega_{\hor}(M)[2].$$
\end{itemize}
\end{corollary}
\begin{proof}
Both statements are immediate consequences of Lemma \ref{lem:degrees}.
\end{proof}

In Section \ref{section: characteristic foliation}, we identify the \lione $(\Omega_\hor(M)/F^2\Omega(M))[2]$ with the \lione that controls the deformations of the characteristic foliation $K=\ker(\eta)$ of the pre-symplectic manifold $(M,\eta)$.
\\

\subsection{Examples}\label{section: examples}

We present two examples for Corollary \ref
{corollary: MC for bivector fields}.

\begin{example}[An example with quadratic Maurer-Cartan equation]
On the open subset $(\RR\setminus \{-1\})\times \RR^2$ of $\RR^3$ 
consider the Poisson bivector field $Z=\pd{y}\wedge\pd{z}$. The closed $2$-form $\alpha:=(xdy+y dx)\wedge dz$ lies in $\mathcal{I}_{-Z}$. We check that $\beta:=F^{-1}(\alpha)$ satisfies 
\begin{equation*}
d\beta+\frac{1}{2}[\beta,\beta]_Z
=0,
\end{equation*}
i.e.~that $\beta[2]$ is a  Maurer-Cartan element of the $L_{\infty}[1]$-algebra of
Proposition \ref{proposition: extended Koszul}, as predicted by Corollary \ref
{corollary: MC for bivector fields}.
We have  $\beta^\sharp=\alpha^\sharp(\id - Z^{\sharp}\alpha^\sharp)^{-1}$ by Remark \ref{rem:alfabeta}, whence 
$$\beta=\frac{1}{1+x}\alpha=\frac{x}{1+x}dy\wedge dz+\frac{y}{1+x}dx\wedge dz.$$
We are done by computing 
\begin{eqnarray*}
d\beta&=&-\frac{x}{(1+x)^2}dx\wedge dy\wedge dz, \\
\, [\beta,\beta]_Z &=& 2\cL_{Z}\beta\wedge \beta=2\frac{x}{(1+x)^2}dx\wedge dy\wedge dz.
\end{eqnarray*}
\end{example}

\begin{example}[An example with cubic Maurer-Cartan equation]
Let $a$ be a smooth function on the real line. 
On the open subset $U$ of $\RR^4$ consisting of points $(x,y,z,w)$ so that $1+a(y)\neq 0$,
let $Z=\pd{y}\wedge(\pd{z}-a(y)\pd{x})$.
From now on we write $a$ instead of $a(y)$ for the ease of notation. Notice that 
$Z$ is a Poisson structure if{f} $G$ is involutive, which happens exactly when the
derivative $a'$ vanishes (i.e.~ $a$ is locally constant).

The closed 2-form $\alpha:= dx\wedge dy+dz\wedge dw$ lies in $\mathcal{I}_{-Z}$. We check that $\beta:=F^{-1}(\alpha)$ satisfies 
\begin{equation}\label{MCsigns}
d\beta+\frac{1}{2}[\beta,\beta]_Z
-\frac{1}{6} (\beta^\sharp\wedge \beta^\sharp\wedge\beta^\sharp )(\frac{1}{2}[Z,Z])
=0,
\end{equation}
i.e.~that $\beta[2]$ is a  Maurer-Cartan element of the $L_{\infty}[1]$-algebra of
Proposition \ref{proposition: extended Koszul}, as prediced by Corollary \ref
{corollary: MC for bivector fields}.

We know that $\beta^\sharp=\alpha^\sharp(\id - Z^{\sharp}\alpha^\sharp)^{-1}$ by Remark \ref{rem:alfabeta}. In matrix form w.r.t. the frames $\pd{x},\pd{y},\pd{z},\pd{w}$ and 
$dx,dy,dz,dw$ we have
$$
Z^\sharp\alpha^\sharp=\left(\begin{array}{c  c   c c}  -a & 0 & 0 & 0  \\ 0& -a & 0 & 1  \\  1 & 0 & 0 &0 \\0 & 0 &0 & 0 \end{array}\right),
\quad\quad
(\mathrm{id}-Z^\sharp \alpha^\sharp)^{-1} =  \frac{1}{1+a}\left(\begin{array}{c  c   c c}  1 & 0 & 0 &0 \\ 0& 1 & 0 &1  \\    1 & 0 & 1+a & 0 \\ 0 &0 & 0 & 1+a \end{array}\right)$$
and therefore
$$\beta= \frac{1}{1+a}(dx\wedge dy+dx\wedge dw)+dz\wedge dw.$$
We compute \begin{equation}\label{eq:cub1}
d\beta=\frac{a'}{(1+a)^2}dx\wedge dy\wedge dw.
\end{equation}
Next we compute $[\beta,\beta]_Z$.
We have $\beta\wedge\beta=\frac{2}{1+a}dx\wedge dy\wedge dz\wedge dw$, so
$$\cL_{Z}(\beta\wedge \beta)=-d\Big(\iota_{\pd{y}}\iota_{\pd{z}-a\pd{x}}\big(\beta\wedge\beta\big)\Big)=2\frac{a'}{(1+a)^2}(dx\wedge dy\wedge dw+dy\wedge dz\wedge dw).$$
Further $\cL_{Z}\beta=\iota_Zd\beta-d\iota_Z\beta=\frac{a'}{(1+a)^2}(dy-adw)$, so that
$$-2\cL_{Z}\beta\wedge \beta=2\frac{a'}{(1+a)^2}(dx\wedge dy\wedge dw-dy\wedge dz\wedge dw).$$ Therefore there is a cancellation and 
\begin{equation}\label{eq:cub2}
[\beta,\beta]_Z=-\cL_{Z}(\beta\wedge \beta)+2\cL_{Z}\beta\wedge \beta=-4\frac{a'}{(1+a)^2}(dx\wedge dy\wedge dw). 
\end{equation}

Finally, using $[Z,Z]=-2a'\pd{x}\wedge\pd{y}\wedge\pd{z}$ we see that
\begin{equation}\label{eq:cub3}
(\beta^\sharp\wedge \beta^\sharp\wedge\beta^\sharp )(\frac{1}{2}[Z,Z])
=-6\frac{a'}{(1+a)^2}dx\wedge dy\wedge dw.
\end{equation}
Using Equations \eqref{eq:cub1}, \eqref{eq:cub2},
\eqref{eq:cub3}, we see that 
the left-hand side of Equation \eqref{MCsigns} reads
$$[1+\frac{1}{2}\cdot(-4)-\frac{1}{6}\cdot(-6)]\frac{a'}{(1+a)^2}dx\wedge dy\wedge dw=0.$$
\end{example}

 \addtocontents{toc}{\protect\mbox{}\protect}
\section{\textsf{The characteristic foliation}}
\label{section: characteristic foliation}

Recall that for any pre-symplectic structure $\eta$
on $M$, the kernel $K=\ker(\eta)$ is an involutive distribution of constant rank. This observation yields a map
from the space of pre-symplectic structures on $M$ of some given rank $k$
to the space of codimension $k$ foliations, which we denote by
$$ \rho: \Presym^k(M) \to \Foliations^k(M).$$

Our aim is to understand this map from an algebraic perspective.

\subsection{A chain map}\label{subsection: foliations - tangent map}

We first recall the Bott-complex of a foliation.
\begin{definition}\label{def:BottConn}
Let $K$ be a foliation of $M$. Denote the normal bundle $TM/K$ by $\nu_K$. It comes equipped with a natural flat, partial connection, called
the Bott-connection, given by
\begin{eqnarray*}
\Gamma(K)\times \Gamma(\nu_K) & \to & \Gamma(\nu_K),\\
(X,Y) & \mapsto & [X, \tilde{Y}] \,\, \mathrm{mod} \, K,
\end{eqnarray*}
where $\tilde{Y}$ is a lift of $Y$ to a vector field on $M$.

The \emph{Bott-complex} of $K$ is the graded vector space
$\Omega(K,\nu_K):=\Gamma(\wedge K^* \otimes \nu_K)$,
equipped with the derivative corresponding to the Bott-connection.
\end{definition}

It is known that the formal tangent space to
the space of foliations at $K$ can be identified with closed elements
of degree $1$
in the Bott-complex. This and Remark \ref{rem:tpres} 
suggest} that for any given pre-symplectic form $\eta$ on $M$, there should exist a chain map $\Omega_\hor(M) \to \Omega(K,\nu_K)$. We now construct this map.
 
\begin{lemma}
\begin{enumerate}
\item Let $K^\circ$ be the dual bundle to $\nu_K$, equipped with the dual (partial) connection. The map $\eta^\sharp: TM \to T^*M$ restricts to a vector bundle isomorphism $\nu_K \cong K^\circ$, which is compatible with the flat connections.
\item The following map is compatible with the differentials:
\begin{eqnarray*}
\varphi: \Omega_\hor^k(M) &\to & \Omega^{k-1}(K,K^\circ)\\
\beta &\mapsto & \varphi(\beta)(X_1,\dots,X_{k-1}):=\beta(X_1,\dots, X_{k-1},\cdot).
\end{eqnarray*}
\end{enumerate}
\end{lemma} 

\begin{proof}

Concerning item (1), the only (slightly) non-obvious fact is the compatibility with the connections. For all $X$, $Y \in \Gamma(\nu_K)$ and $W\in \Gamma(K)$ we need to verify the equality
$$ \langle\nabla_W \eta^\sharp(X), Y\rangle \overset{!}{=} \langle\eta^\sharp(\nabla_W X),Y\rangle = \eta(\nabla_W X,Y).$$
By definition of the connection on $K^\circ$, the left-hand side can be written as
$$ W(\eta(X,Y)) - \eta(X,\nabla_W Y).$$
In total, we see that compatibility of $\eta^\sharp$ with the connections is equivalent to the equation
$$ W(\eta(X,Y)) = \eta([W,\tilde{X}],Y) + \eta(X,[W,\tilde{Y}]),$$
where $\tilde{X}$ and $\tilde{Y}$ are extensions of $X$ and $Y$ to vector fields.
This equation is satisfied because it is precisely given by the non-vanishing terms on the left-hand side of $d\eta(W,\tilde{X},\tilde{Y})=0$. 

For item (2), we have to check for all $X_1,\dots, X_k \in \Gamma(K)$
and $Y \in \Gamma(\nu_K)$, that the equality
$$ \langle\varphi(d\beta)(X_1,\dots,X_k),Y\rangle = \langle d_\nabla(\varphi(\beta))(X_1,\dots,X_k),Y\rangle$$
holds.
By definition, the left-hand side is equal to
$ d\beta(X_1,\dots,X_k,\tilde{Y})$, where $\tilde{Y}$ is an extension of $Y$ to a vector field, while the right-hand side is equal to 
\begin{eqnarray*}
 &&\sum_{i=1}^k(-1)^{i-1} \langle\nabla_{X_i}(\beta(X_1,\dots,\hat{X}_i,\dots,X_k,\cdot)),\tilde{Y}\rangle \\
 &&+ \sum_{i<j}(-1)^{i+j}\beta([X_i,X_j],\dots,\hat{X}_i,\dots,\hat{X}_j,\dots,\tilde{Y}),
 \end{eqnarray*}
where the first line equals
\begin{eqnarray*}
\sum_{i=1}^k(-1)^{i-1}X_i(\beta(X_1,\dots,\hat{X}_i,\dots,X_k,\tilde{Y}))
 + \sum_{i=1}^k (-1)^i \beta(X_1,\dots,\hat{X}_i,\dots,X_k,[X_i,\tilde{Y}]).
\end{eqnarray*}
Using that $\beta$ vanishes when restricted to $K$, we see that, in total, the above expression sums up to
$ d\beta(X_1,\dots,X_k,\tilde{Y})$ as well.
 \end{proof}
 
 \begin{corollary}\label{rem:Bottchainmap}
 The following composition is a chain map
 
\begin{equation}\label{eq:comp}
\xymatrix{
 \Omega^\bullet_\hor(M) \ar[r]^(0.45){\varphi} & \Omega^{\bullet-1}(K,K^\circ) \ar[r]^{(\eta^\sharp)^{-1}} & \Omega^{\bullet-1}(K,\nu_K).
 }
\end{equation}
 \end{corollary}

\begin{remark}
 One can check that the linear map
$$\{\alpha \in \Omega^2_\hor(M),\, d\alpha=0\} \to \{\beta \in \Omega^1(K,\nu_K),\, d_{\nabla}\beta=0\}$$
obtained by restricting \eqref{eq:comp}
coincides with {\em minus} the formal tangent map of
$\rho:  \Presym^k(M) \to \Foliations^k(M)$
at $\eta \in \Presym^k(M)$.
\end{remark}

\subsection{Deformation theory of foliations}
\label{subsection: deformations of foliations}

We first review the $L_{\infty}[1]$-algebra governing deformations of foliations.
We follow the exposition of  Xiang Ji \cite[Theorem 4.20]{Ji} (who works in the wider setting of deformations of Lie subalgebroids). For foliations, these (and more) results were first obtained by Huebschmann \cite{Huebsch}  and   Vitagliano \cite{VitaglianoFol} (see \cite[Section 8]{VitaglianoFol} and  \cite[Rem. 4.23]{Ji} for a comparison of results).

\begin{proposition}\label{prop:Ji}
Let $K$ be an involutive distribution on a manifold $M$, and let $G$ be a complement.
There is an $L_{\infty}[1]$-algebra structure on
 $\Gamma(\wedge K^*\otimes G)[1]$, whose only non-vanishing brackets are $l_1,-l_2,l_3$, with the property that the graph of $\phi\in \Gamma(K^*\otimes G)[1]$ is involutive if{f}  $\phi$ is a Maurer-Cartan element.
\end{proposition}

 We remark that $l_1$ is the differential associated to the flat $K$-connection\footnote{Explicitly, this is the flat $K$-connection of $G$ given by $\nabla_XY=\mathrm{pr}_G[X,Y]$ for all $X\in \Gamma(K), Y\in \Gamma(G)$.}  on $G$ which, under the identification $G\cong TM/K$, corresponds to the Bott connection from Section \ref{subsection: foliations - tangent map}. 
  The formulae for $l_1,l_2,l_3$ are spelled out in the following remark.

\begin{remark}\label{rem:Jisproof}
Ji \cite[Section 4]{Ji} derives this $L_{\infty}[1]$-multibrackets $l_1,l_2,l_3$ as follows: he views $TM=G\oplus K$ as a vector bundle over $K$ (with fibres isomorphic to those of $G$), and 
 applies  Voronov's derived bracket construction (see Appendix \ref{appendix: reminder L-infty}) to:
\begin{itemize}
\item  the graded Lie algebra $V$ of vector fields on $T[1]M$,
\item the abelian Lie subalgebra $\mathfrak{a}$  of vector fields which are  vertical and fibrewise constant with respect to the projection $T[1]M\to K[1]$,
\item the   Lie subalgebra  $\mathfrak{h}$ of vector fields that are tangent to the base manifold $K[1]$,
\item  the homological vector field $X$ on $T[1]M$ encoding the de~Rham differential on $M$ (notice that $X$ is tangent to $K[1]$ because $K$ is an involutive distribution). 
\end{itemize}
This construction delivers the multibrackets $l_1,l_2,l_3$ in the formulae below. To relate these formulas to the formulas obtained in \cite[Remark 4.16]{Ji}, one has to apply the isomorphism\footnote{The necessity to apply this isomorphism is related to the following piece of linear algebra. Let $V$ be a vector space, and consider $V[1]$ (a graded manifold with linear coordinates of degree $1$). The functions on $V[1]$ coincide with the exterior algebra $\wedge V^*$, and the vector fields on $V[1]$ coincide with $\wedge V^*\otimes V$ (in particular, degree $-1$ vector fields on $V[1]$ are just elements of $V$). Now take a $k$-multilinear skew-symmetric map on $V$ with values in $V$, i.e.~an element
$\alpha\in \wedge^kV^*\otimes V$. This element can also be viewed canonically as a (degree $k-1$) vector field on $V[1]$, which we denote by $X$. Consider the $k$-multilinear skew-symmetric map on $V$ with values in $V$ given by
$$v_1,\dots,v_k\mapsto [[\dots[X,\iota_{v_1}],\dots],\iota_{v_k}],$$
where $\iota_{v_i}$ is the (degree $-1$) vector field on $V[1]$ given by $v_i\in V$, and the bracket is the graded Lie bracket of vector fields on $V[1]$. 
The point is that this map \emph{differs from $\alpha$ by a factor of $(-1)^\frac{k(k+1)}{2}$}, as one can easily check in coordinates, and this factor is omitted  in \cite{Ji}.} 
of $\Gamma(\wedge K^*\otimes G)$ which acts on each homogeneous component $\Gamma(\wedge^k K^*\otimes G)$ via multiplications by $(-1)^\frac{k(k+1)}{2}$.

For all $\xi \in \Gamma(\wedge^k K^*\otimes G)[1], \psi \in \Gamma(\wedge^l K^*\otimes G)[1],
 \phi \in \Gamma(\wedge^m K^*\otimes G)[1]$ we have:

\begin{eqnarray*}
 l_1(\xi)(X_1,\dots,X_{k+1})&=&\sum_{i=1}^{k+1}(-1)^{i+1} \mathrm{pr}_G\Big[X_i\,,\,\xi(X_1,\dots,\hat{X_i},\dots,X_{k+1})\Big]\\
 &+& \sum_{i<j}(-1)^{i+j}\xi\Big([X_i,X_j],X_1,\dots,\hat{X_i},\dots,\hat{X_j},\dots,X_{k+1}\Big)\\ 
  l_2(\xi,\psi)(X_1,\dots,X_{k+l})&=&(-1)^{k}\sum_{\tau\in S_{k,l}}(-1)^{\tau} \mathrm{pr}_G\Big[\xi(
 X_{\tau(1)},\dots,X_{\tau(k)})\,,\,\psi(X_{\tau(k+1)},\dots,X_{\tau(k+l)})\Big]
 \\
  &+&(-1)^{k(l+1)} \sum_{\tau\in S_{l,1,k-1}}(-1)^{\tau} \xi\Big(\mathrm{pr}_K\Big[\psi( X_{\tau(1)},\dots,X_{\tau(l)})\,,\,X_{\tau(l+1)}\Big],\dots\\
  && \hspace{6cm} \dots,X_{\tau(l+2)},\dots,  X_{\tau(l+k)}\Big)
 \\
  &-& (-1)^{k} (\xi \leftrightarrow \psi, k\leftrightarrow l)\end{eqnarray*}
  \begin{align*}
 l_3(\xi,\psi,\phi)(X_1,\dots,&X_{k+l+m-1})=  (-1)^{m+k(l+m)}\cdot\\
\sum_{\tau\in S_{l,m,k-1}}(-1)^{\tau}  \xi\Big(\mathrm{pr}_K &\Big[
 \psi(X_{\tau(1)},\dots,X_{\tau(l)})\,,\,\phi(X_{\tau(l+1)},\dots,X_{\tau(l+m)})\Big],
,\dots,X_{\tau(l+m+k-1)}\Big) \pm\circlearrowleft
\end{align*}
Here $X_i\in \Gamma(K)$,  $\mathrm{pr}_G$ is the projection $TM=G\oplus K\to G$, and similarly for $\mathrm{pr}_K$. $S_{i,j,k}$ denotes the set of permutations $\tau$ of $i+j+k$ elements such that the order is preserved within each block:
$\tau(1)<\cdots<\tau(i), \tau(i+1)<\cdots<\tau(i+j), \tau(i+j+1)<\cdots<\tau(i+j+k)
$. The symbol
$(\xi \leftrightarrow \psi, k\leftrightarrow l)$ denotes the sum just above it, switching $\xi$ with $\psi$ and $k$ with $l$. The symbol $\circlearrowleft$ denotes cyclic permutations in $\xi,\psi,\phi$.
\end{remark}

\begin{proof}[Proof of Proposition \ref{prop:Ji}]
Given Remark \ref{rem:Jisproof}, we just have to address the statement about Maurer-Cartan elements, which holds by \cite[Theorem 4.14]{Ji}.
\end{proof}

 The $L_{\infty}[1]$-algebra structure is compatible with the 
$\Gamma(\wedge K^*)$-module structure, as described for instance in \cite[Theorem 2.4 and 2.5]{VitaglianoFol}. We spell this out:
 
\begin{lemma}\label{lem:Jider}
With respect to the (left) $\Gamma(\wedge K^*)$-module structure on $\Gamma(\wedge K^*\otimes G)[1]$,
the multibrackets of the $L_{\infty}[1]$-algebra structure described in Proposition \ref{prop:Ji} have the following properties:  
\begin{align*}
  l_1(\alpha\cdot \xi)&=(d\alpha)|_{\wedge K}\cdot \xi+ (-1)^{a}\alpha \cdot l_1(\xi)\\
 l_2( \xi, \alpha\cdot \psi)&=(-1)^{k}(\iota_{\xi} d\alpha)|_{\wedge K}\cdot \psi+ (-1)^{ka}\alpha \cdot l_2(\xi,\psi)\\
  l_3(\xi, \psi, \alpha\cdot  \phi)&=(-1)^{k+l+1}\iota_{\xi}\iota_{\psi} d\alpha\cdot \phi + (-1)^{(k+l+1)a}\alpha \cdot l_3(\xi, \psi,\phi).
\end{align*}
for all $\alpha \in \Gamma(\wedge^a K^*)$ and 
 $\xi \in \Gamma(\wedge^k K^*\otimes G)[1], \psi \in \Gamma(\wedge^l K^*\otimes G)[1],
 \phi \in \Gamma(\wedge K^*\otimes G)[1]$.
 \end{lemma}
\begin{proof}
This follows from a computation using the characterization of the  multibrackets in terms of  the graded Lie algebra of vector fields on $T[1]M$, as described in Remark \ref{rem:Jisproof}.
Alternatively, one can derive it from the explicit formulas for $l_1$, $l_2$ and $l_3$ given above.
\end{proof}

\begin{remark}[The involutive case]
When $G$ is involutive, the trinary bracket $l_3$ vanishes and, after a degree shift, we obtain a dg Lie algebra.
In that case, {around every point of $M$ one can find a local frame of $K$ consisting of}  
sections $X\in \Gamma(K)$ whose flow preserves $G$ (i.e.~$[X,\Gamma(G)]\subset \Gamma(G)$).
If $X_i$ are such sections we have
$$ l_2(\xi,\psi)(X_1,\dots,X_{k+l})=(-1)^{k(l+1)}\sum_{\tau\in S_{k,l}}(-1)^{\tau} \mathrm{pr}_G\Big[\xi(
 X_{\tau(l+1)},\dots,X_{\tau(l+k)})\,,\,\psi(X_{\tau(1)},\dots,X_{\tau(l)})\Big]
$$
for any $\xi \in\Gamma(\wedge^k K^*\otimes G)[1]$ and $\psi \in\Gamma(\wedge^l K^*\otimes G)[1]$. 
 
In the involutive case it is easy to check directly the following claim made in Proposition \ref{prop:Ji}: 
the graph of $\phi\in \Gamma(K^*\otimes G)[1]$ is involutive if{f} $\phi$ satisfies the Maurer-Cartan equation of $(\Gamma(\wedge K^*\otimes G)[1],l_1,-l_2,l_3)$, i.e.~if{f}
$$l_1(\phi)-\frac{1}{2}l_2(\phi,\phi)=0.$$ Indeed, given $X_1,X_2$ as above, writing out $[X_1+\phi(X_1),X_2+\phi(X_2)]$
we obtain $[X_1,X_2]$ plus 3 terms lying in $\Gamma(G)$. Hence $\mathrm{graph}(\phi)$ is involutive if{f} 
 $$0=-\phi([X_1,X_2])+[X_1,\phi(X_2)]-[X_2,\phi(X_1)]+[\phi(X_1),\phi(X_2)].$$
 The first 3 terms combine to $(l_1(\phi))(X_1,X_2)$, while the last term is 
$-\frac{1}{2}l_2(\phi,\phi)(X_1,X_2)$. \end{remark}

\subsection{A strict morphism of \liones}
\label{subsection: strict morphism}

As earlier, let $\eta$ be a pre-symplectic form on $M$, and choose a complement $G$ to 
$K=\ker(\eta)$. Let $Z \in\Gamma(\wedge^2 G)$ be the bivector field corresponding to the restriction of $\eta$ to $G$, so $Z^{\sharp}:=-(\eta^\sharp|_G)^{-1}\colon G^*\to G$. Recall from Corollary \ref{cor:liideal} that $F^{2}(\Omega(M))=\Omega_\hor(M)\cdot \Omega_\hor(M)$ gives an $L_{\infty}[1]$-ideal of $\Omega_\hor(M)[2]$. This suggests the following result, where we use the notation $\Omega(K,G):=\Gamma(\wedge K^*\otimes G)$ and similarly for $\Omega(K,G^*)$:
\begin{theorem}\label{theorem: characteristic foliation}
 The composition 
$$ q[2]: \Omega_\hor(M)[2] \rightarrow \faktor{\Omega_\hor(M)[2]}{F^2(\Omega(M))[2]} \cong \Omega(K,G^*)[1] \stackrel{Z^{\sharp}[1]}{\longrightarrow} \Omega(K,G)[1]$$
is a strict morphism of $L_\infty[1]$-algebras,
where the domain is the Koszul $L_{\infty}[1]$-algebra with multibrackets 
$\lambda_1,\lambda_2,\lambda_3$, see Theorem \ref{theorem: construction - Koszul L-infty}, and
 the target $\Omega(K,G)[1]$ is endowed with the multibrackets
$l_1,-l_2,l_3$.
\end{theorem}
\begin{remark}
\begin{itemize}
\item[i)] For every element of $\Omega_\hor(M)=\Gamma(\wedge^{\bullet} K^*\otimes \wedge^{\ge 1} G^*)$, the first map in the above composition 
simply selects the component in  
$\Omega(K,G^*)=\Gamma(\wedge^{\bullet} K^*\otimes G^*)$.
Hence $q$ maps a Maurer-Cartan element $(\mu,\sigma)\in  \Gamma(K^*\otimes G^*) \oplus\Gamma(\wedge^2G^*)$ to the Maurer-Cartan element $Z^{\sharp}\mu^{\sharp}\in \Gamma(\mathrm{Hom}(K,G))$,  where $\mu^{\sharp}\colon K\to G^*, X\mapsto \iota_X\mu$.
Notice that, via Theorem \ref{theorem: main result} and Proposition \ref{prop:Ji}, this is consistent with the fact that the kernel of $(\mu,\sigma)$ is the graph of $Z^{\sharp}\mu^{\sharp}$ (see Theorem \ref{theorem: almost Dirac structures}) and with the well-known fact that the kernels of pre-symplectic forms are involutive.
\item[ii)] A consequence of Theorem \ref{theorem: characteristic foliation} is that the map
$$Z^{\sharp}[2]: \faktor{\Omega_\hor(M)[2]}{F^2\Omega(M)[2]} \cong \Omega(K,G^*)[1] \longrightarrow \Omega(K,G)[1]$$
is a strict isomorphism of $L_{\infty}[1]$-algebras. Here
$\faktor{\Omega_\hor(M)[2]}{F^2(\Omega(M))[2]}$ is endowed with the $L_{\infty}[1]$-algebra structure inherited from $\Omega_\hor(M)[2]$. 
\end{itemize}
\end{remark}

We need some preparation before giving the proof of Theorem \ref{theorem: characteristic foliation}. Notice that for $\alpha \in \Omega(K)$ and $\xi\in \Gamma(G^*)$,
the element $\alpha \cdot  \xi\in \Gamma(\wedge K^*\otimes G^*)$ is mapped to $\alpha \cdot Z^{\sharp} \xi\in \Omega(K,G)$ under $q$. In the next three statements, we write $\sharp$ instead of $Z^{\sharp}$,
and suppress the degree shifts $[2]$ for the sake of readability.

\begin{proposition}\label{prop:intertwo}
 For all $\xi_1,\xi_2\in \Gamma(G^*)$ and $\alpha_1,\alpha_2\in \Gamma(\wedge K^*)$ we have
  
$$q\left(\lambda_2 (\alpha_1\cdot\xi_1,\alpha_2\cdot\xi_2)\right)=-l_2(\alpha_1 \cdot\sharp \xi_1,\alpha_2 \cdot\sharp \xi_2).$$

\end{proposition}
\begin{proof}
Using the biderivation property of the Koszul bracket $[\cdot,\cdot]_Z$ (see Remark \ref{rem:Kos2der}) and Lemma \ref{lem:Jider} for $l_2$, we obtain
\begin{align*}
q([\alpha_1\xi_1,\alpha_2\xi_2]_Z)&=\alpha_1[ \xi_1,\alpha_2]_Z|_{\wedge K}\cdot \sharp \xi_2
-(-1)^{|\alpha_1||\alpha_2|}\alpha_2[ \xi_2,\alpha_1]_Z|_{\wedge K}\cdot \sharp \xi_1
+\alpha_1\alpha_2\cdot \sharp [\xi_1,\xi_2]_Z\\
(-1)^{|\alpha_1|}l_2(\alpha_1\sharp \xi_1,\alpha_2 \sharp \xi_2)&=\alpha_1(\iota_{\sharp \xi_1} d\alpha_2)|_{\wedge K}\cdot \sharp \xi_2
-(-1)^{|\alpha_1||\alpha_2|}\alpha_2(\iota_{\sharp \xi_2} d\alpha_1)|_{\wedge K}\cdot \sharp \xi_1
+\alpha_1\alpha_2 \cdot l_2(\sharp \xi_1,\sharp \xi_2).
\end{align*}

Since $\lambda_2 (\alpha_1\cdot\xi_1,\alpha_2\cdot\xi_2)=-(-1)^{|\alpha_1|}[\alpha_1\xi_1,\alpha_2\xi_2]_Z$, we have to prove that the two expressions above coincide.
 The right-most terms coincide by Lemma \ref{lem:morG*}, hence it suffices to prove that 
$[ \xi_1,\alpha_2]_Z=\iota_{\sharp \xi_1} d\alpha_2$. This identity is proven by induction over the degree of $\alpha_2$, using that the case in which  $\alpha_2$ is a 1-form   holds by Equation \eqref{eq:KoszulBracket1f}, page \ref{eq:KoszulBracket1f}.\end{proof}

 Now we consider the trinary brackets.

\begin{lemma}\label{lem:Cartancalc}
For all $\xi_1,\xi_2\in \Gamma(G^*)$ and $\alpha \in \Gamma(\wedge K^*)$ we have
$\lambda_3(\xi_1,\xi_2,\alpha)
=\iota_{\sharp \xi_2}\iota_{\sharp \xi_1}d\alpha$.
\end{lemma}
\begin{proof}
By the definition in Proposition \ref{proposition: extended Koszul}  we have
$\lambda_3(\xi_1,\xi_2,\alpha)
=\frac{1}{2}(\xi_1^{\sharp}\wedge\xi_2^{\sharp}\wedge\alpha^{\sharp}) {[Z,Z]}$. 
Now
$$\iota_{\sharp \xi_2}\iota_{\sharp \xi_1}d\alpha=-\iota_{[\sharp \xi_1,\sharp \xi_2]}\alpha
=\frac{1}{2}\iota_{(\iota_{\xi_2}\iota_{\xi_1}[Z,Z])}\alpha,$$ 
using some Cartan calculus and the fact that $\sharp \xi_i\in \Gamma(G)$ in the first equality
and Equation \eqref{eq:sharpnothom}, page \pageref{eq:sharpnothom} in the second. Using the fact that $[Z,Z]$ is a section of $\wedge^2G\otimes K$ (see Section \ref{section: bivec}), this concludes the proof.
\end{proof}
 
\begin{proposition}\label{prop:interthree}
 For all $\xi_1,\xi_2,\xi_3\in \Gamma(G^*)$ and $\alpha_1,\alpha_2,\alpha_3\in \Gamma(\wedge K^*)$ we have
\begin{equation}\label{eq:q3}
q\;(\lambda_3(\alpha_1\cdot\xi_1,\alpha_2\cdot\xi_2,\alpha_3\cdot\xi_3))=l_3(\alpha_1 \sharp \xi_1,\alpha_2 \sharp \xi_2,\alpha_3 \sharp \xi_3).
\end{equation}
\end{proposition}

\begin{proof}
Using repeatedly the derivation property of $\lambda_3$ (Remark \ref{rem:Kos3der}),
the fact that $[Z,Z]\in \Gamma(\wedge^2G\otimes K)$ (see Section \ref{section: bivec}), 
 and the fact that $\lambda_3$ is graded symmetric, we obtain
\begin{eqnarray}\label{eq:lambdaleib}
\lambda_3(\alpha_1\cdot\xi_1,\alpha_2\cdot\xi_2,\alpha_3\cdot\xi_3)&=&
(-1)^{|\alpha_1|}\alpha_1\alpha_2\lambda_3(\xi_1,\xi_2,\alpha_3)\xi_3\\
&&+(-1)^{|\alpha_2||\alpha_3|+|\alpha_1|+|\alpha_2|+|\alpha_3|}
\alpha_1\alpha_3\lambda_3(\xi_3,\xi_1,\alpha_2)\xi_2\nonumber\\
&&+(-1)^{|\alpha_1||\alpha_2|+|\alpha_1||\alpha_3|+|\alpha_3|}
\alpha_2\alpha_3\lambda_3(\xi_2,\xi_3,\alpha_1)\xi_1.\nonumber
\end{eqnarray}
Applying repeatedly Lemma \ref{lem:Jider} for $l_3$, using the fact that  $l_3$ applied to any three elements of $\Gamma(G)$ vanishes (by its very definition), and using three times Lemma \ref{lem:Cartancalc},
 we see that 
$l_3(\alpha_1 \sharp \xi_1,\alpha_2 \sharp \xi_2,\alpha_3 \sharp \xi_3)$ equals the image under $q$ of the right-hand side of Equation \eqref{eq:lambdaleib}. 
\end{proof}

\begin{proof}[Proof of Theorem \ref{theorem: characteristic foliation}]
Clearly the kernel of $q$ is $F^2(\Omega(M))$, and
 each $\lambda_i$ maps to zero in the quotient $\faktor{\Omega_\hor(M)}{F^2(\Omega(M))}$ whenever one of the entries lies in 
$F^2\Omega(M)$, by  Lemma \ref{lem:degrees}. Hence, to check that the map 
$q$ is a strict $L_{\infty}[1]$-algebra morphism,
it is sufficient to restrict ourselves to elements of $\Gamma(\wedge K^*\otimes G^*)$. The restriction of $q$ there 
\begin{itemize}
\item intertwines $\lambda_1$ and $l_1$  by Corollary \ref{rem:Bottchainmap} since    the identification $\nu_K\cong G$  (given by the projection along $K$) identifies $l_1$ and the Bott connection,
\item  intertwines $\lambda_2$ and $-l_2$ by Proposition \ref{prop:intertwo}, 
\item intertwines $\lambda_3$ and $l_3$ by Proposition \ref{prop:interthree}.
\end{itemize}
 \end{proof}

 \addtocontents{toc}{\protect\mbox{}\protect}
\section{\textsf{Obstructedness of deformations}}
\label{section: obstructions}

We display examples showing that both the deformations of pre-symplectic forms and of foliations are formally obstructed. A deformation problem governed by an $L_{\infty}[1]$-algebra $(W,\{\lambda_k\}_{k\ge 1})$ is called {\em formally obstructed}, if there is a class in the zero-th cohomology $H^0(W)$ of the cochain complex $(W,\lambda_1)$, such that one (or equivalently, any) representative $w\in W_0$ can not be extended to a formal curve of Maurer-Cartan elements.  
Obstructedness can be detected with the help of the Kuranishi map
\begin{equation}
\mathrm{Kr} \colon H^0(W)  \to H^1(W), \;\;\ [w]\mapsto [\lambda_2(w,w)],
\end{equation}
for it is well-known that if the Kuranishi map is not identically zero, then the deformation problem is formally (and hence also smoothly) obstructed, see for example \cite[Theorem 11.4]{OP}.

\subsection{Obstructedness of pre-symplectic deformations}
\label{subsec:obspre}

We consider  
  $M=(S^1)^{\times 4}$ with  pre-symplectic form $\eta:=d\theta_3 \wedge d\theta_4$ (so the kernel $K$ is spanned by $\frac{\partial}{\partial \theta_1}$, $\frac{\partial}{\partial \theta_2}$), and choosing
as complementary subbundle $G$ the span of
$\frac{\partial}{\partial \theta_3},\frac{\partial}{\partial \theta_4}$ 
(so the Poisson bivector field corresponding to $\eta$ reads
$\pi:=\frac{\partial}{\partial \theta_3}\wedge \frac{\partial}{\partial \theta_4}$).

In Section \ref{subsection: Koszul for pre-symplectic} we saw that the Koszul $L_{\infty}[1]$-algebra ($\Omega_\hor(M)[2],\lambda_1,\lambda_2,\lambda_3)$ governs the deformations of $\eta$ as a pre-symplectic form. Recall that the multibrackets $\lambda_i$ were defined in Proposition \ref{proposition: extended Koszul}, with $\lambda_1$ being the de Rham differential.
 
The Kuranishi map  
$\mathrm{Kr}\colon H^0\to H^1$ maps the class of any closed element $B\in \Omega^2_\hor(M)$ to the class of $[B,B]_{\pi}$. 
Now we choose $$B:=f(\theta_3)d\theta_1\wedge d\theta_3+g(\theta_4) d\theta_2\wedge d\theta_4$$ for $f,g$ smooth functions on $S^1=\RR/2\pi \ZZ$. Clearly $B$ is closed. One computes\footnote{To see this, notice that $[d\theta_i,d\theta_j]_{\pi}=0$. Further,  for $h\in \mathcal{C}^\infty(M)$, one has
{
\begin{eqnarray*}
[d\theta_i,h]_\pi&=&\Lie_\pi(d\theta_i h)=  [\iota_\pi,d](d\theta_i h)
= -\iota_\pi(d\theta_i \wedge dh) = -\iota_{\frac{\partial}{\partial \theta_3}}\iota_{\frac{\partial}{\partial \theta_4}}(d\theta_i dh)
= \delta_{i3} \frac{\partial h}{\partial \theta_4} - \delta_{i4} \frac{\partial h}{\partial \theta_3}.
\end{eqnarray*} } 
}
 that $$ [B,B]_{\pi}=2\big(f{\partial_4 g}d\theta_1\wedge d\theta_2\wedge d\theta_4+
g{\partial_3 f}d\theta_1\wedge d\theta_2\wedge d\theta_3),$$
where $\partial_i$ denotes the partial derivative w.r.t. $\theta_i$.
While this 3-form is exact, we now show that it does not admit a primitive in $\Omega^2_\hor(M)$. This will show that the class it represents in the cohomology of 
$\Omega_\hor(M)$ is non-zero, i.e. that $\mathrm{Kr}([B])\neq 0$.

For any $\alpha \in \Omega^2_\hor(M)$, we compute the   integral of $d\alpha$
over $C_{[a,b]}:={S^1\times S^1\times [a,b] \times \{0\}}$ for all values of $a,b\in S^1$, where $[a,b]$ denotes the positively oriented arc joining these two points:
$\int_{C_{[a,b]}}d\alpha=
\int_{\partial C_{[a,b]}}\alpha=0$,
using Stokes' theorem and the fact that $\alpha$ has no $d\theta_1\wedge d\theta_2$ component.
On the other hand, $$ \int_{C_{[a,b]}}[B,B]_{\pi}= 2(2\pi)^2g(0)[f(b)-f(a)],$$ which is non-vanishing for instance for $f=g=\cos$. Hence $[B,B]_{\pi}$ can not be equal to $d\alpha$ for any $\alpha \in \Omega^2_\hor(M)$.

\subsection{Obstructedness of deformations of foliations}
As above take the manifold $M=(S^1)^{\times 4}$, take $K$ the involutive distribution  spanned by $\frac{\partial}{\partial \theta_1}$, $\frac{\partial}{\partial \theta_2}$, and choose
as complementary subbundle $G$  the span of
$\frac{\partial}{\partial \theta_3},\frac{\partial}{\partial \theta_4}$. 
As seen in Section \ref{subsection: deformations of foliations}, the deformations of the involutive distribution $K$ (i.e. of the underlying foliation) are goverened by the 
 $L_{\infty}[1]$-algebra 
 $(\Gamma(\wedge K^*\otimes G)[1],l_1,-l_2,l_3)$. 
 
 The Kuranishi map  
$\mathrm{Kr}\colon H^0\to H^1$ maps the class of  any $l_1$-closed $\Phi\in \Gamma(K^*\otimes G)$ to the class of $-l_2(\Phi,\Phi)$.  
We now take $$\Phi=d\theta_1\otimes f(\theta_3)\frac{\partial}{\partial \theta_4}-
d\theta_2  \otimes g(\theta_4)\frac{\partial}{\partial \theta_3},$$
which is $l_1$-closed by Equation \eqref{eq:lixi} below.
We compute $l_2(\Phi,\Phi)\in \Gamma(\wedge^2 K^*\otimes G)$ by evaluating it on the frame $\frac{\partial}{\partial \theta_1}, \frac{\partial}{\partial \theta_2}$ of $K$ and obtain 
$$l_2(\Phi,\Phi)= (d\theta_1\wedge d\theta_2) \otimes 2\big(f{\partial_4 g}\frac{\partial}{\partial \theta_3}- g {\partial_3 f}\frac{\partial}{\partial \theta_4}\big).$$
On the other hand,  any $\xi\in \Gamma(K^*\otimes G)$ can be written as $\sum_{i=1}^2\sum_{j=3}^4 d\theta_i\otimes h_{ij}
\frac{\partial}{\partial \theta_j}$ for functions $h_{ij}$ on $M=(S^1)^{\times 4}$, and one computes
\begin{equation}\label{eq:lixi}
l_1(\xi)= (d\theta_1\wedge d\theta_2) \otimes\big[({\partial_1}h_{23}- {\partial_2}h_{13})\frac{\partial}{\partial \theta_3}+
({\partial_1}h_{24}- {\partial_2}h_{14})\frac{\partial}{\partial \theta_4}\big].
\end{equation}
Notice that for every $a,c\in S^1$, the pullback of   $l_1(\xi)$
to $N_{a,c}:=S^1\times S^1\times \{a\}\times \{c\}$ is exact\footnote{A primitive is $(h_{13}d\theta_1+h_{23}d\theta_2)\otimes \frac{\partial}{\partial \theta_3}+ 
(h_{14}d\theta_1+h_{24}d\theta_2)\otimes \frac{\partial}{\partial \theta_4}$.}, hence
$\int_{N_{a,c}}l_1(\xi)=0$ for all $a,c$.
But the pullback of $l_2(\Phi,\Phi)$ is a constant 2-form, since $f$ and $g$ do not depend on $\theta_1,\theta_2$, hence $\int_{N_{a,c}}l_2(\Phi,\Phi)$ 
is a (vector-valued) constant. This constant 
 is non-zero for instance for $f=g=\cos$ and for $a,c\notin \frac{\pi}{2}\ZZ$. It follows that  $l_2(\Phi,\Phi)$ is not equal to 
$l_1(\xi)$ for any $\xi\in \Gamma(K^*\otimes G)$. In other words, it follows that  $\mathrm{Kr}([\Phi])\neq 0$.
\begin{remark}
$\Phi$ is the image of $B$ under the strict $L_{\infty}[1]$-morphism $q$ of Theorem \ref{theorem: characteristic foliation}. Hence the induced map in cohomology maps $\mathrm{Kr}([B])$ to $\mathrm{Kr}([\Phi])$. This and the fact that $\mathrm{Kr}([\Phi])\neq 0$ implies that 
the result we obtained in Subsection \ref{subsec:obspre}, namely that $\mathrm{Kr}([B])\neq 0$.
\end{remark}

\appendix
\addtocontents{toc}{\protect\setcounter{tocdepth}{1}}

 \addtocontents{toc}{\protect\mbox{}\protect}
\section{Cartan calculus of multivector fields}
\label{sec:Cartan}

We recall the Cartan calculus on manifolds. Let $M$ be a manifold and $Y$ a multivector field on $M$ of degree $k$. Associated with $Y$, we have the following operators on $\Omega(M)$:
\begin{itemize}
\item {\em Contraction}: $\iota_Y: \Omega^\bullet(M)\to \Omega^{\bullet-k}(M)$, which for $Y=f$ a function is ordinary multiplication, for a vector field $X$ is defined by
$$(\iota_X \omega)(X_1,\cdots, X_{r-1}):=\omega(X,X_1,\cdots,X_{r-1}),$$
and then extended to all multivector fields by the rule
$$\iota_{X\wedge \tilde{X}} = \iota_X \circ \iota_{\tilde{X}}.$$
\item {\em Lie derivative}: $\Lie_Y: \Omega^\bullet(M) \to \Omega^{\bullet - k +1}(M)$, which is defined as the graded commutator $[\iota_Y,d]=\iota_Y\circ d - (-1)^{|Y|}d\circ \iota_Y$, with $d$ the de~Rham differential.
\end{itemize}

These operations obey the following commutator rules:
\begin{enumerate}
\item $[\Lie_Y,d]=0$ and $[\iota_Y,\iota_{\tilde{Y}}]=0$,
\item $[\Lie_Y,\iota_{\tilde{Y}}]=\iota_{[Y,\tilde{Y}]}$, for $[\cdot,\cdot]$ the Schouten-Nijenhuis bracket of multivector fields, and
\item $[\Lie_Y,\Lie_{\tilde{Y}}]=\Lie_{[Y,\tilde{Y}]}$.
\end{enumerate}

\section{Reminder on $L_\infty$-algebras, higher derived brackets, and Koszul brackets}
\label{appendix: reminder L-infty}

\subsection{$L_\infty$- and $L_\infty[1]$-algebras}

We briefly review the basic background about $L_\infty$- and $L_\infty[1]$-algebras.
Let $V$ be a graded vector space.

\begin{itemize}
	\item For every $r\in \mathbb{Z}$, we have the degree shift endofunctor $[r]$, which maps a graded vector space $V$
	to $V[r]$, whose component $(V[r])^i$ in degree $i\in \mathbb{Z}$ is $V^{i+r}$.
	\item We denote by $S_n$ the symmetric group on $n$ letters. Given an integer $n\ge1$ and an ordered partition $i_1+\cdots+i_k=n$, we denote by $S_{i_1,\ldots,i_k}\subset S_n$ the set of \emph{$(i_1,\ldots,i_k)$-unshuffles}, i.e., permutations $\sigma\in S_n$ such that $\sigma(i)<\sigma(i+1)$ for $i\neq i_1, i_1+i_2,\ldots, i_1+\cdots+ i_{k-1}$.
	\item $\overline{T}(V)=\bigoplus_{n\ge 1}T^n(V)$ is the $n$-fold tensor product of $V$ with itself. The symmetric group $S_n$ acts on $T^n(V)$ by $\sigma(x_1\otimes\cdots\otimes x_n)=\varepsilon(\sigma)x_{\sigma(1)}\otimes\cdots\otimes x_{\sigma(n)}$, where $\varepsilon(\sigma)=\varepsilon(\sigma;x_1,\ldots,x_n)$ is the usual \emph{Koszul sign}. 
	We denote the space of coinvariants by $\bigodot^nV$, and by $x_1\odot\cdots\odot x_n$ the image of $x_1\otimes\cdots\otimes x_n$ under the natural projection $T^n(V)\to \bigodot^n V$. 
	\item The reduced symmetric coalgebra over $V$ is the space $\bigodot V =\bigoplus_{n\ge 1}\bigodot^n V$, with the unshuffle coproduct
	$$\overline{\Delta}(x_1\odot \cdots \odot x_n) = \sum_{i=1}^{n-1}\sum_{\sigma\in S_{i,n-i}}\varepsilon(\sigma)(x_{\sigma(1)}\odot \cdots \odot x_{\sigma(i)})\otimes (x_{\sigma(i+1)}\odot \cdots \odot x_{\sigma(n)}).$$
	This is the cofree, coassociative, cocommutative and locally conilpotent graded coalgebra over $V$.
	\item 
	Let $(C,\Delta)$ be a graded coalgebra.
	A map $Q: (C,\Delta)\to (C,\Delta)$ of degree $1$ is a codifferential if $Q\circ Q = 0$ and $\Delta \circ Q =(Q\otimes \mathrm{id} + \mathrm{id}\otimes Q)\circ \Delta$ hold true.
	\item An $L_\infty[1]$-algebra structure on $V$ is a codifferential $Q$
	of the graded coalgebra $(\bigodot V,\overline{\Delta})$.
	\item A morphism of $L_\infty[1]$-algebras from $L_\infty[1]$-algebra $V$ to $L_\infty[1]$-algebra $W$ is a morphism of the corresponding
	dg coalgebras $F: (\bigodot V,\overline{\Delta},Q_V) \to (\bigodot W,\overline{\Delta},Q_W)$.
	\item An $L_\infty[1]$-algebra structure $Q$ on $V$ is determined by its Taylor coefficients $(Q_n)_{n\ge 1}$, which are the maps given by
	$$ 
	\xymatrix{
		\bigodot^n V \ar[r] & \bigodot V \ar[r]^Q & (\bigodot V)[1] \ar[r]^(0.6){p}& V[1].
	}$$
	Moreover, a morphism $F$ of $L_\infty[1]$-algebras from $V$ to $W$
	is determined by its Taylor coefficients
	$F_n: \bigodot^n V \to W$, which are defined in the same manner as
	the Taylor coefficients of an $L_\infty[1]$-algebra structure.
	\item A morphism of $L_\infty[1]$-algebras is called an isomorphism if the corresponding morphism of dg coalgebras is invertible.	It is called strict if all its structure maps except for the first one vanish.
	\item A graded subspace $W$ of an $L_\infty[1]$-algebra $V$ is an $L_\infty[1]$-subalgebra if the corresponding structure map $Q_n$ sends $\odot^n W$ to $W$.
	Similarly, $W$ is an $L_\infty[1]$-ideal if $Q_n$ maps
	$W\odot \odot^{n-1}V$ to $W$.
	\item The category of dg Lie algebras embeds into the category of $L_\infty[1]$-algebra via
	$$ (L,d,[\cdot,\cdot]) \mapsto (\bigodot(L[1]),Q),$$
	where $Q$ is the coderivation whose non-trivial Taylor coefficients are
	$Q_1( a[1]) = - (da)[1]$ and $Q_2( a[1]\otimes  b[1])=(-1)^{|a|}([a, b])[1]$.
	\item Finally, we define the structure of an $L_\infty$-algebra  on $V$ to be an $L_\infty[1]$-algebra structure on $V[1]$.
\end{itemize}

The main example of an $L_\infty$-algebra (respectively $L_\infty[1]$-algebra) in this paper is the {\em Koszul $L_\infty$-algebra ($L_\infty[1]$-algebra)} associated to a pre-symplectic manifold, cf.~Section \ref{subsection: Koszul for pre-symplectic}.

\subsection{Higher derived brackets}

The formalism of higher derived brackets from \cite{Voronov1,Voronov2} is a mechanism to construct $L_\infty$-algebras from certain input data.
The input data are a graded Lie algebra $(V,[\cdot,\cdot])$,
together with
\begin{itemize}
\item a splitting as a graded vector space $V=\mathfrak{a}\oplus \mathfrak{h}$, where $\mathfrak{a}$ is an abelian subalgebra and $\mathfrak{h}$ is a Lie subalgebra,
\item a Maurer-Cartan element $X \in V$, i.e.~an element of degree $1$ such that $[X,X]=0$ holds.
\end{itemize}
There is a compatibility condition, which requires that 
$X$ lies in $\mathfrak{h}$.

The higher derived brackets associated to these data are the maps defined by
$$ Q_n: \odot^n\mathfrak{a} \to \mathfrak{a}[1], \quad a_1\odot \cdots \odot a_n \mapsto \mathrm{pr}_{\mathfrak{a}}\Big([\cdots [X,a_1],\cdots,a_n]\Big),$$
where $\mathrm{pr}_{\mathfrak{a}}$ denotes the projection from $V$ to $\mathfrak{a}$ along $\mathfrak{h}$.

The following result is proven in \cite{Voronov1}:

\begin{theorem}\label{theorem: higher derived brackets}
The maps $(Q_n)_{n\ge 1}$ equip $\mathfrak{a}$ with the structure of an $L_\infty[1]$-algebra.
\end{theorem}

\subsection{$BV_\infty$-algebras and Koszul brackets}
\label{appendix: BV}

We collect some useful facts from the literature about commutative $BV_\infty$-algebras and Koszul brackets. We follow mostly
\cite{Koszul,Kravchenko} and the exposition in \cite[Section 4.2.1]{Bandiera_PhD-thesis}.
We refer the reader for details and proofs to these sources.

First recall the notion of differential operators on a graded commutative algebra $A$: One defines recursively the set $\Diffop_k(A) \subset \mathrm{End}(A)$ of differential operators
of order $\le k$ on $A$ by
$$ \Diffop_{-1}(A)=\{0\}, \quad \textrm{and} \quad \Diffop_k(A) =\{f\in \Hom(A,A) \, \vert \, [f,\ell_a]\subset \Diffop_{k-1}(A) \, \forall a \in A\},$$
where $\ell_a$ denotes left multiplication, i.e.~$\ell_a(b):=ab$.

Following \cite{Kravchenko} we define

\begin{definition}\label{definition: comm BV-infty}
A \emph{commutative $BV_\infty$-algebra} $(A,d=\Delta_0,\Delta_1,\dots)$
of (odd) degree $r$ is a commutative, unital dg algebra $(A,d,\cdot)$,
equipped with a family of endomorphisms $(\Delta_i)_{i\ge 0}$
of degree $1-i(r+1)$ such  that:
\begin{enumerate}
\item for all $i\ge 1$, the endomorphism $\Delta_i$ is a differential operator of order $\le i+1$,
which annihilates the unit $1_A$,
\item if we adjoin a central variable $t$ of degree $r+1$, the operator
$$ \Delta:= \Delta_0 + t \Delta_1 + t^2 \Delta_2+\cdots: A[[t]]\to A[[t]]$$
squares to zero.
\end{enumerate}
We refer to $\Delta$ as the BV-operator.
\end{definition}

It is well-known that to every commutative $BV_\infty$-algebra of degree $r$,
one can associate an $L_\infty[1]$-algebra structure on $A[r+1]$, see 
\cite{Voronov1,Bandiera_PhD-thesis}. 
The Taylor coefficients of this $L_\infty[1]$-algebra structure are given by
the Koszul brackets associated to the operators $\Delta_i$:
One associates to the endomorphism $\Delta_i$ a sequence
of operations $\Koszul(\Delta_i)_n: \odot^n A\to A$ defined iteratively by
$\Koszul(\Delta_i)_1=\Delta_i$ and
\begin{eqnarray*}
\Koszul(\Delta_i)_n(a_1\odot \cdots \odot a_n) &=& +\Koszul(\Delta_i)_{n-1}(a_1\odot \cdots \odot a_{n-2}\odot a_{n-1}a_n)\\
&& - \Koszul(\Delta_i)_{n-1}(a_1\odot \cdots \odot a_{n-1})a_n\\
&&-(-1)^{|a_{n-1}||a_n|}\Koszul(\Delta_i)_{n-1}(a_1\odot \cdots \odot a_{n-1}\odot a_n)a_{n-1},
\end{eqnarray*}

{ The following result was noticed in \cite{Voronov1}, we follow the exposition from \cite[Proposition 4.2.21]{Bandiera_PhD-thesis}:}

\begin{proposition}\label{proposition: Koszul brackets}
Given a commutative $BV_\infty$-algebra $(A,d,=\Delta_0,\Delta_1,\dots,)$ of degree $r$, the sequence of maps
$$(\Delta_0,\Koszul(\Delta_1)_2,\Koszul(\Delta_2)_3,\dots)$$
equips $A[r+1]$ with the structure of an $L_\infty[1]$-algebra.
\end{proposition}

\bibliographystyle{habbrv} 

\end{document}